\documentclass[final,reqno]{siamltex}
\usepackage{latexsym,amsmath,amssymb,amsfonts,mathrsfs}
\usepackage{epsf,graphicx,epsfig,color,cite,cases}
\usepackage{subfigure,graphics,multirow,marginnote,enumerate,bm}
\newcommand{\Grad}{{\rm Grad\,}}

\newcommand{\Om}{\Omega}
\newcommand{\om}{\omega}

\newcommand{\pa}{\partial}

\newcommand{\ov}{\overline}
\newcommand{\I}{{\rm Im}}
\newcommand{\Rt}{{\rm Re}}

\newcommand{\dive}{{\rm div\,}}

\newcommand{\wid}{\widetilde}

\newcommand{\na}{\nabla}
\newcommand{\mat}{\mathbb}

\newcommand{\R}{{\mat R}}

\newcommand{\N}{{\mat N}}

\newcommand{\Sp}{{\mat S}}
\newcommand{\ds}{\displaystyle}

\newcommand{\be}{\begin{eqnarray}}
\newcommand{\ben}{\begin{eqnarray*}}
\newcommand{\en}{\end{eqnarray}}
\newcommand{\enn}{\end{eqnarray*}}

\newtheorem{lem}[theorem]{Lemma}
\newtheorem{assumption}[theorem]{Assumption}
\newtheorem{remark}[theorem]{Remark}

\begin{document}
\renewcommand{\theequation}{\arabic{section}.\arabic{equation}}

\title{\bf On uniqueness of inverse conductive scattering problem with unknown embedded obstacles}
\author{Chengyu Wu\thanks{School of Mathematics and Statistics, Xi'an Jiaotong University,
Xi'an 710049, Shaanxi, China ({\tt wucy99@stu.xjtu.edu.cn})}
\and
Jiaqing Yang\thanks{School of Mathematics and Statistics, Xi'an Jiaotong University,
Xi'an 710049, Shaanxi, China ({\tt jiaq.yang@xjtu.edu.cn})}
}
\date{}
\maketitle


\begin{abstract}
 This paper is concerning the inverse conductive scattering of acoustic waves by a bounded inhomogeneous object with possibly embedded obstacles inside. A new uniqueness theorem is proved that the conductive object is uniquely determined by the fixed frequency far-field measurements,  ignoring its contents. Meanwhile, the boundary informations of several related physical coefficients are also uniquely determined. The proof is mainly based on a detailed singularity analysis of solutions near the interface associated with a family of point sources or hypersingular point sources, which is deduced by the potential theory. Moreover, the other key ingredient in the proof is the well-posedness of the interior transmission problem with the conductivity boundary condition in the $L^2$ sense, where several sufficient conditions depending on the domain and physical coefficients are provided.
\end{abstract}



\pagestyle{myheadings}
\thispagestyle{plain}
\markboth{C. Wu and J. Yang}{Uniqueness in inverse acoustic conductive scattering}

\section{Introduction}\label{sec1}
\setcounter{equation}{0}
In this paper, we consider the inverse conductive scattering problem of determining penetrable object with possibly embedded obstacles. Such kind of problems can find applications in various fields, such as remote sensing, radar and sonar, medical imaging and nondestructive testing. 

Mathematically, denote by $D$ and $D_b$ the bounded penetrable object in $\R^3$ and the embedded obstacles, respectively, where we assume that $D,D_b\subset\R^3$ are both bounded domains such that $\pa D\in C^{2,\alpha_0}$ and $D_b\subset\subset D$, $D_b=\bigcup_{j=1}^m D_b^{(j)}$ with $D_b^{(j_1)}\cap D_b^{(j_2)}=\emptyset$ if $j_1\neq j_2$. Denote by $n$ the refractive index in the inhomogeneous medium $\R^3\setminus\overline{D_b}$, which is assumed that $n=1$ in $\R^3\setminus\overline{D}$, $n\in L^\infty(D\setminus\overline{D_b})$, $\Rt(n)>0$ and $\I(n)\geq0$ in $D\setminus\overline{D_b}$. Consider an incident wave $u^i$ satisfying the Helmholtz equation $\Delta u^i+k^2 u^i=0$ in $\R^3$ with the wave number $k>0$ given by $k=\om/c$, where $\om>0$ is the frequency and $c>0$ is the sound speed. Then the scattering problem corresponding to the scatterer $(D,\lambda,\gamma,n,D_b)$ is modeled by the following transmission problem
\be\label{1.1}
  \left\{
  \begin{array}{ll}
    \Delta u+k^2u=0~~~&{\rm in}~\R^3\setminus\ov D,\\
    \dive(\lambda\nabla v)+k^2nv=0~~~&{\rm in}~D\setminus\overline{D_b},\\ 
    u-v=0~~~&{\rm on}~\pa D,\\ 
    \pa_\nu u=\lambda\pa_\nu v+\gamma v~~~&{\rm on}~\pa D,\\ 
    \mathscr{B}(v)=0~~~&{\rm on}~\pa D_b,\\ 
    \displaystyle\lim\limits_{r\rightarrow\infty}r\left(\pa_r u^s-iku^s\right)=0,~&r=|x|, 
  \end{array}
  \right.
\en
where $\lambda\in C^1(\ov D\setminus D_b)$ with $\lambda>0$ in $\ov D\setminus D_b$ is the transmission coefficient related to the medium $\R^3\setminus\overline{D_b}$, $\gamma\in C(\pa D)$ is the conductive coefficient with $\I(\gamma)\leq0$, $u^s$ is the scattered wave, $u=u^i+u^s$ denotes the total field in $\R^3\setminus\ov D$, and $\nu$ is the unit exterior normal to $\pa D$. Furthermore, $\mathscr{B}$ stands for the boundary condition imposed on $\pa D_b$ satisfying $\mathscr{B}(v)=v$, provided that  $D_b$ is a sound-soft obstacle, and $\mathscr{B}(v)=\pa_{\nu}v$ with $\nu$ the unit exterior normal to $\pa D_b$, provided that $D_b$ is a sound-hard obstacle. Moreover, if $D_b$ is a mixed-type obstacle then $\mathscr{B}(v)=\pa_\nu v+\beta v$ on an open set $\Gamma_b\subset\pa D_b$(with the impedance coefficient $\beta$ such that $\I(\beta)\geq0$) and $\mathscr{B}(v)=v$ on $\pa D_b\setminus\ov\Gamma_b$.

The Sommerfeld radiation condition leads to the following asymptotic behavior 
\ben
  u^s(x)=\frac{e^{ikr}}{4\pi r}\left\{u_{\infty}(\hat{x})+O\left(\frac{1}{r}\right)\right\},~~~~r=|x|\rightarrow\infty,
\enn
for the scattered field $u^s$ uniformly in all directions $\hat{x}=x/r\in\Sp^2$, where $u_\infty$ is the far-field pattern defined on the unit sphere. The well-posedness has been extensively studied in the literature, and we refer to e.g.  \cite{FD06,DR13,AL98,XB10,JBH18}. In this paper, we study the inverse problem to determine $D,\lambda,n,\gamma$ and $D_b$ from a knowledge of the far-field pattern at a fixed frequency. 

Many uniqueness theorems are available in the literature for the case of $\lambda\in\R^+$, $\gamma=0$ and $D_b=\emptyset$. The first result is initiated by Isakov \cite{VI90} with the variational method by constructing special singular solutions for two groups of Helmholtz equations. Then Kirsch \& Kress \cite{AR93} simplified Isakov's approach by utilizing the potential theory in the classical continuous function spaces. 
From then on, lots of inverse problems for different mathematical models have been investigated (see, e.g., \cite{DM98,FH94,TR96,VI08,XB10,NV04,JB11,FH96,DM98,PH93,FD03,XBJ10} and the references therein). However, it is noted that among these work, there are rare that can handle the case $D_b\neq\emptyset$ with no other strong assumptions, and their methods are usually complicated. For example, the penetrable obstacle $D$ and the embedded object $D_b$ can be determined simultaneously in \cite{XB10} but under the condition that $n$ is a known constant in $D\setminus\ov D_b$. In \cite{PH98}, H\"{a}hner get the uniqueness of $D$ and $D_b$ by considering the completeness of the products of the harmonic functions, which is, however, with the knowledge of the scattered waves corresponding to an interval of wave numbers. It should also be reminded that the majority of the above uniqueness results were under the hypothesis $\lambda\neq1$ near the boundary $\pa D$. 

When focusing on the conductive boundary condition, that is, $\gamma\neq0$, the related work are much less and almost all are only considering constant $\lambda$. In \cite{FH94}, also based on Isakov's idea, Hettlich obtained the uniqueness of $D,n,\lambda$ and $\gamma|_{\pa D}$, under the assumptions that $n,\lambda$ are constant in $D$ and $D_b=\emptyset$ with other certain conditions. For the case $\lambda=1$, Hettlich gave an additional assumption that $\gamma$ has a positive distance from 0 on $\pa D$ in order to get the uniqueness, which is not necessary in our results. Later, Gerlach \& Kress \cite{TR96} applied the same approach in \cite{AR93} to shorten the analysis of Hettlich. For the case $D_b\neq\emptyset$ or $n,\lambda$ are not constant, there is barely no scientific literature. We here only mention that Mitrea and Mitrea \cite{DM98} get the uniqueness of $\pa D$ with possibly $D_b\neq\emptyset$ but also with constant $n$ and $\lambda$, whose methods are based on boundary integral techniques and Calder\'{o}n-Zygmund theory. 

As discussed above, the singular solution for the scattering problem plays a key role in the inverse problem. Thus in the present paper we study in detail the solutions to problem \eqref{1.1} with point source or hypersingular point source incident waves. It is a natural thought that the solutions possess a same singularity as the singular incidence when the singular positions approach the interface. Here we give the first rigorous proof for such thought. We show that the solutions have the same singularity as the singular incidence multiplying by a constant relative to the transmission coefficient $\lambda$ and conductive coefficient $\gamma$. We establish these results by spilting the singular solution into several parts. Different methods have been applied to analyze different parts, such as the variational method and the integral equation method, and for the key part it requires an elaborate mathmatical analysis. The detailed singularity for singular solutions, besides the uniqueness results of the inverse problem, is another ingredient of this paper. 

For the inverse problem, it is well known that $D_b\neq\emptyset$ or $\lambda=1$ can bring great difficulties in the uniqueness proof of the inverse problem. But recently, a much simpler method for the inverse transmission problem is proposed in \cite{JBH18}, which can recover the boundary $\pa D$ in situations that $D_b\neq\emptyset$ or $\lambda=1$. This method is applied in the present paper and extended to many other cases, like in \cite{FJB17} for the inverse fluid-solid interaction problem and in \cite{JY21} for the inverse boundary value problem. It is expected to be used in more cases. Our method is mainly based on constructing a well-posed modified interior transmission problem in a sufficiently small domain inside $D$ near the boundary $\pa D$. We use the smallness of the domain to eliminate the difficulties generated by $D_b\neq\emptyset$ or $\lambda=1$ near the boundary $\pa D$. Further, our uniqueness results contain the identification of the derivatives of the coefficients on the boundary. The proof is greatly relied on the aforementioned complete description of the singularity of the solutions to the scattering problem with singular incidence. We note that these estimates are possible to be generalized to determine the higher order derivatives of the coefficients. And it is also possible to extend our methods to other types inverse scattering problems, such as the electromagnetic and elastic scattering. 

The paper is organized as follows. In Section \ref{sec2}, we propose a novel perspective to derive the singularity of solutions with point source or hypersingular point source incident waves in the $H^1$ sense, which plays a key role in the uniqueness proof of the inverse problem. In Section \ref{sec3.1}, we introduce two types interior transmission problem corresponding to the case $\lambda\neq1$ and $\lambda=1$ near the boundary, respectively. The new type interior transmission problem for the case $\lambda=1$ is then investigated carefully and several sufficient conditions for it to be well-posed are given. In particular, we find that this interior transmission problem is always uniquely solvable in domain sufficiently small,  which is essential for later proof. Finally, Section \ref{sec3.2} is all about the proof of the uniqueness result. A novel and simple method is employed here for the unique determination of the penetrable object from the measurement of the far-field pattern at a fixed frequency, disregarding its contents. Furthermore, we also identify the coefficients in the scattering problem at the boundary of the penetrable object utilizing the singularity of the solutions.

\section{Singularity analysis}\label{sec2}
\setcounter{equation}{0}
This section is devoted to the study of the singularity of the solutions to problem \eqref{1.1}, as the source position approaching the interface. It will be proved that the solutions admit the same singularity to the singular incidence multiplying by certain constants related to $\lambda$ and $\gamma$. These singularities play a key role in proving the uniqueness results of the inverse problem and are also interesting on their own rights. 

For the sake of simplicity, we just consider the case with an impedance condition on $\pa D_b$, which is $\mathscr{B}(v)=\pa_\nu v+i\rho v$ with $\rho\geq0$. Other cases can be handled analogously. We start by presenting some useful notations. 

For $x_0\in\pa D$ and $\delta>0$ sufficiently small, define $x_j:=x_0+(\delta/j)\nu(x_0)\in\R^3\setminus\ov D$ and $y_j:=x_0-(\delta/j)\nu(x_0)\in D\setminus\ov D_b$, $j\in\N$. Now we introduce the single- and double-layer boundary operators
\ben
&&(S_{ee}\varphi)(x)=\int_{\pa D}\Phi(x,y)\varphi(y)ds(y),~~~x\in\pa D, \\
&&(K_{ee}\varphi)(x)=\int_{\pa D}\frac{\pa\Phi(x,y)}{\pa\nu(y)}\varphi(y)ds(y),~~~x\in\pa D,
\enn
and their normal derivative operators
\ben
&&(K'_{ee}\varphi)(x)=\int_{\pa D}\frac{\pa\Phi(x,y)}{\pa\nu(x)}\varphi(y)ds(y),~~~x\in\pa D,\\
&&(T_{ee}\varphi)(x)=\frac{\pa}{\pa\nu(x)}\int_{\pa D}\frac{\pa\Phi(x,y)}{\pa\nu(y)}\varphi(y)ds(y),~~~x\in\pa D.
\enn
where $\Phi(x,y)=\exp(ik|x-y|)/(4\pi|x-y|)$ is the fundamental solution to the Helmholtz equation in $\mathbb{R}^3$. We also introduce the boundary operators $S_{ii}$, $K_{ii}$, $K'_{ii}$ and $T_{ii}$ defined on $\pa D_b$ as well as $S_{th}$, $K_{th}$, $K'_{th}$ and $T_{th}$ with $t,h=e,i$, respectively, where, for example, $S_{ei}$ is defined similarly as $S_{ee}$ but with $x\in\pa D_b$. It is known that $S_{tt}$, $K_{tt}$ and $K'_{tt}$ with $t=e,i$ are all bounded and compact in $L^q(\pa D)(1<q<\infty)$. We refer the readers to \cite{RP01,RP98} for the properties of these operators in $L^p$ spaces. 
\begin{theorem}\label{thm2.8}
	Denote by $(u_j^s,v_j)$ the solution for problem {\rm (\ref{1.1})} with the incident wave $u_j^i(x)=\Phi(x,x_j)$, $j\in\N$. Then, 
	\ben
	\left\|v_j-\frac{2}{\lambda(x_0)+1}\Phi_0(x,x_j)\right\|_{H^{1}(D\setminus\ov D_b)}\leq C
	\enn
	uniformly for $j\in\N$ with $C>0$ a constant, where $\Phi_0(x,y)=1/(4\pi|x-y|)$ is the fundamental solution for the Laplacian. 
\end{theorem}
\begin{proof}
	We prove by spilting the solutions $(u_j^s,v_j)$ into two parts. First, we consider the case with constant $\lambda$ and $n$. Define $(\hat u_j^s,\hat v_j)$ to be the unique solution of the following equations 
	\be\label{2.12}
	\left\{
	\begin{array}{ll}
		\Delta\hat u^s_j+k^2\hat u^s_j=0~~~&{\rm in}~\R^3\setminus\ov D,\\
		\dive(\lambda_0\nabla\hat v_j)+k^2\hat v_j=0~~~&{\rm in}~D\setminus\overline{D_b},\\
		\hat u^s_j-\hat v_j=f_{1,j}~~~&{\rm on}~\pa D,\\ 
	    \pa_\nu\hat u^s_j-\left(\lambda_0\pa_\nu\hat v_j+\gamma\hat v_j\right)=f_{2,j}~~~&{\rm on}~\pa D,\\ 
		\mathscr{B}(\hat v_j)=0~~~&{\rm on}~\pa D_b,\\
		\displaystyle\lim\limits_{r\rightarrow\infty}r\left(\pa_r\hat u^s_j-ik\hat u^s_j\right)=0,~&r=|x|, 
	\end{array}
	\right.
	\en
	where $\lambda_0=\lambda(x_0)$ and the boundary datas 
	\ben
	&&f_{1,j}(x)=-(\Phi(x,x_j)+\Phi(x,y_j)), \\ 
	&&f_{2,j}(x)=-\left(\frac{\pa\Phi(x,x_j)}{\pa\nu(x)}+\frac{\pa\Phi(x,y_j)}{\pa\nu(x)}\right).
	\enn
	Clearly, $(u_j^s-\Phi(\cdot,y_j),v_j)$ solves problem \eqref{2.12} replacing $\lambda_0$ and $k^2$ by $\lambda$ and $k^2n$. We will analyze $\hat v_j$ and $v_j-\hat v_j$ separately. 
	
	Before going further, we note that $f_{1,j}\in L^p(\pa D)$ with $1\leq p<2$ is uniformly bounded for $j\in\N$, and by \cite[Lemma 4.2]{DRP97} $f_{2,j}$ is uniformly bounded in $C(\pa D)$ for $j\in\N$. 
Set $k_1^2=k^2/\lambda_0>0$. Now let $(\hat u_j^s,\hat v_j)$ be in the form of 
  \be\label{2.15a}
    \hat u_j^s(x)\;&&=\int_{\pa D}\Phi(x,y)\varphi_j(y)ds(y)+\lambda_0\int_{\pa D}\frac{\pa\Phi(x,y)}{\pa\nu(y)}\psi_j(y)ds(y),~~x\in\R^3\setminus\ov D,\\
    \label{2.16a}\nonumber
    \hat v_j(x)\;&&=\int_{\pa D}\Phi_1(x,y)\varphi_j(y)ds(y)+\int_{\pa D}\frac{\pa\Phi_1(x,y)}{\pa\nu(y)}\psi_j(y)ds(y)\\
    \;&&\quad\;+\int_{\pa D_b}\Phi_1(x,y)\eta_j(y)ds(y),~~x\in D\setminus\ov D_b 
  \en
  with $\Phi_1(x,y)=\exp(ik_1|x-y|)/(4\pi|x-y|)$. Then by the jump relations of these operators, it is equivlently reduced to the following system of boundary integral equations
  \be\label{2.2}
    \left(
      \begin{array}{ccc}
        1 & 0 & 0 \\
        -h\gamma & 1 & 0 \\
        0 & 0 & 1 \\
      \end{array}
    \right)
    \left(
      \begin{array}{c}
        \psi_j \\
        \varphi_j \\
        \eta_j
      \end{array}
    \right)
    +L
    \left(
      \begin{array}{c}
        \psi_j \\
        \varphi_j \\
        \eta_j
      \end{array}
    \right)
    =
    \left(
      \begin{array}{c}
        2hf_{1,j} \\
        -2hf_{2,j} \\
        0
      \end{array}
    \right)
  \en
  in $L^p(\pa D)\times L^p(\pa D)\times C(\pa D_b)$ with $p<2$,
  where
  $
    L=(L_1,L_2,L_3)
  $
  with 
  \ben
    &&L_1=\left(2h(\lambda_0 K_{ee}-K_{ee}^{(1)}),2h\gamma K_{ee}^{(1)}-2h\lambda_0(T_{ee}-T_{ee}^{(1)}),-2(T_{ei}^{(1)}+i\rho K_{ei}^{(1)})\right)^T,\\
    &&L_2=\left(2h(S_{ee}-S_{ee}^{(1)}),2h(\lambda_0 K_{ee}^{'(1)}+\gamma S_{ee}^{(1)}-K'_{ee}),-2(K_{ei}^{'(1)}+i\rho S_{ei}^{(1)})\right)^T, \\
    &&L_3=\left(-2hS_{ie}^{(1)},2h(\lambda_0 K'_{ie}+\gamma S_{ie}^{(1)}),-2(K_{ii}^{'(1)}+i\rho S_{ii}^{(1)})\right)^T
  \enn
  and $h=1/(\lambda_0+1)$. The operators above $S_{th}^{(1)}$, $K_{th}^{(1)}$, $K_{th}^{'(1)}$ and $T_{th}^{(1)}$ with $t,h=e,i$ are defined similarly as $S_{th}$, $K_{th}$, $K'_{th}$ and $T_{th}$ with the kernel function $\Phi_1(x,y)$. Since the operators in $L$ are all compact in corresponding Banach spaces, we have that (\ref{2.2}) is of Fredholm type. Thus, the existence of solution $(\psi_j,\varphi_j,\eta_j)^T\in L^p(\pa D)\times L^p(\pa D)\times C(\pa D_b)$ to (\ref{2.2}) follows from the uniqueness of the direct problem \eqref{1.1} with the estimate that 
  \ben
    \|\psi_j\|_{L^p(\pa D)}+\|\varphi_j\|_{L^p(\pa D)}+\|\eta_j\|_{L^\infty(\pa D_b)}\leq C(\|f_{1,j}\|_{L^p(\pa D)}+\|f_{2,j}\|_{L^p(\pa D)})\leq  C. 
  \enn 

  Since $L^p(\pa D)$ is bounded embedded into $H^{-1/2}(\pa D)$ for $p>4/3$ in the two-dimensional case, by the first equation in (\ref{2.2}) and \cite[Corollary 3.7]{DR13}, we see that
  \be\label{2.19a}
    \left\|\psi_j-\frac{2}{\lambda_0+1}f_{1,j}\right\|_{H^{1/2}(\pa D)}\leq C.
  \en
  Moreover, applying \cite[Corollary 3.8]{DR13} we can derive from (\ref{2.16a}) and (\ref{2.19a}) that
  \ben
    &&\left\|\hat v_j-\int_{\pa D}\frac{\pa\Phi_1(x,y)}{\pa\nu(y)}\psi_j(y)ds(y)\right\|_{H^1(D\setminus\ov D_b)}\leq C, \\
    &&\left\|\int_{\pa D}\frac{\pa\Phi_1(x,y)}{\pa\nu(y)}\left(\psi_j(y)-\frac{2}{\lambda_0+1}f_{1,j}(y)\right)ds(y)\right\|_{H^1(D\setminus\ov D_b)}\leq C,
  \enn
  which implies
  \be\label{2.20a}
    \left\|\hat v_j-\frac{2}{\lambda_0+1}\int_{\pa D}\frac{\pa\Phi_1(x,y)}{\pa\nu(y)}f_{1,j}(y)ds(y)\right\|_{H^1(D\setminus\ov D_b)}\leq C
  \en
  uniformly for $j\in\N$.

  Clearly, we have 
  \be\label{2.21a}\nonumber
    &&\quad\int_{\pa D}\frac{\pa\Phi_1(x,y)}{\pa\nu(y)}f_{1,j}(y)ds(y)+2\int_{\pa D}\frac{\pa\Phi_0(x,y)}{\pa\nu(y)}\Phi_0(y,x_j)ds(y)\\ \nonumber
    &&=\int_{\pa D}\frac{\pa(\Phi_1-\Phi_0)(x,y)}{\pa\nu(y)}f_{1,j}(y)ds(y)+\int_{\pa D}\frac{\pa\Phi_0(x,y)}{\pa\nu(y)}(\Phi_0(y,x_j)-\Phi_0(y,y_j))ds(y)\\
    &&\quad\;+\int_{\pa D}\frac{\pa\Phi_0(x,y)}{\pa\nu(y)}\left(f_{1,j}(y)+\Phi_0(y,x_j)+\Phi_0(y,y_j)\right)ds(y).
  \en
  Note that $\Phi_0(\cdot,x_j)-\Phi_0(\cdot,y_j)\in C(\pa D)$ uniformly for $j\in\N$. Denote by $\Grad$ the surface gradient, consider 
  \be\label{2.22a}\nonumber
    &&\quad\Grad(\Phi_0(x,x_j)-\Phi_0(x,y_j))\\ \nonumber
    &&=\na_x[\Phi_0(x,x_j)-\Phi_0(x,y_j)]-(\nu(x)\cdot\na_x[\Phi_0(x,x_j)-\Phi_0(x,y_j)])\nu(x)\\ \nonumber
    &&=\na_x[\Phi_0(x,x_j)-\Phi_0(x,y_j)]-(\nu(x_0)\cdot\na_x[\Phi_0(x,x_j)-\Phi_0(x,y_j)])\nu(x_0)\\ \nonumber
    &&\quad-((\nu(x)-\nu(x_0))\cdot\na_x[\Phi_0(x,x_j)-\Phi_0(x,y_j)])\nu(x)\\
    &&\quad-(\nu(x_0)\cdot\na_x[\Phi_0(x,x_j)-\Phi_0(x,y_j)])(\nu(x)-\nu(x_0)).
  \en
  Direct calculation yields that 
  \ben
    |\na_x(\Phi_0(x,x_j)-\Phi_0(x,y_j))|\leq\frac{C}{|x-x_0|^2},~~~x\in\pa D.
  \enn
  Since $\pa D\in C^2$, we have $|\nu(x)-\nu(x_0)|\leq C|x-x_0|$ for all $x\in\pa D$, which implies the last two terms in the right hand side of  (\ref{2.22a}) are bounded by $C/|x-x_0|$ on $\pa D$. Without loss of generality, we assume that $x_0=(0,0,0)$ and $\nu(x_0)=(0,0,1)$, which further indicates that $|x_{(3)}|\leq C\left(x_{(1)}^2+x_{(2)}^2\right)$ for all $x=(x_{(1)},x_{(2)},x_{(3)})\in\pa D$ by the Taylor's expansion. Then we deduce that 
  \ben
    \frac{|x_{(i)}|}{|x-x_j|}\leq C~~~\quad\text{and}~~~\quad\frac{|x_{(3)}|}{|x-y_j|^2}\leq C 
  \enn
  uniformly for $j\in\N$ on $\pa D$ with $i=1,2$. Moreover, it is derived that 
  \ben
    \frac{1}{j^2|x-y_j|^2}\;&&=\frac{1}{j^2\left(x_{(1)}^2+x_{(2)}^2\right)+j^2x_{(3)}^2-2j\delta x_{(3)}+\delta^2} \\
    &&\leq\frac{1}{Cj^2|x_{(3)}|+j^2x_{(3)}^2-2j\delta x_{(3)}+\delta^2} \\
    &&\leq C,~~~x\in\pa D,
  \enn
  since $\delta>0$ is sufficiently small.
  Now for $x=(x_{(1)},x_{(2)},x_{(3)})\in\pa D$ we obtain that
  \ben
    &&\quad\na_x[\Phi_0(x,x_j)-\Phi_0(x,y_j)]-(\nu(x_0)\cdot\na_x[\Phi_0(x,x_j)-\Phi_0(x,y_j)])\nu(x_0)\\
    &&=\left(\left(\frac{1}{|x-y_j|^3}-\frac{1}{|x-x_j|^3}\right)x_{(1)},\left(\frac{1}{|x-y_j|^3}-\frac{1}{|x-x_j|^3}\right)x_{(2)},0\right).
  \enn
  It is seen that 
  \ben
    &&\quad\left|\left(\frac{1}{|x-y_j|^3}-\frac{1}{|x-x_j|^3}\right)x_{(i)}\right|\\
    &&=\frac{2\delta}{j|x-y_j|}\frac{|x_{(i)}|}{|x-x_j|}\frac{|x_{(3)}|}{|x-y_j|^2}\left(\frac{|x-y_j|}{|x-x_j|^2}+\frac{1}{|x-x_j|+|x-y_j|}\right)\\
    &&\leq \frac{C}{|x-x_0|}
  \enn
  on $\pa D$ with $i=1,2$. Hence, we get 
  \ben
    |\Grad(\Phi_0(x,x_j)-\Phi_0(x,y_j))|\leq\frac{C}{|x-x_0|},~~~x\in\pa D,
  \enn
  which implies $\Grad(\Phi_0(x,x_j)-\Phi_0(x,y_j))\in L^p(\pa D)$ uniformly for $j\in\N$ with $1\leq p<2$ and thus $\Phi_0(x,x_j)-\Phi_0(x,y_j)$ is uniformly bounded in $W^{1,p}(\pa D)$ for $j\in\N$ with $1\leq p<2$. By the classic Sobolev embeddings theorem, we know that the identical map from $W^{1,q}(\pa D)$ into $H^{1/2}(\pa D)$ is bounded for $q>4/3$. Therefore, we derive that the right hand side of (\ref{2.21a}) is uniformly bounded in $H^1(D\setminus\ov D_b)$ from the regularity of $\Phi_1-\Phi_0$ and \cite[Corollary 3.8]{DR13}. 

  Set
  \ben
    \wid f_j(x)=2\int_{\pa D}\frac{\pa\Phi_0(x,y)}{\pa\nu(y)}\Phi_0(y,x_j)ds(y),~~~x\in D.
  \enn
  It is deduced that 
  \ben
  \left\{
  \begin{array}{ll}
    \Delta (\wid f_j+h_j)=0~~~&{\rm in}~D,\\
    \displaystyle\wid f_j(x)+h_j(x)=2\int_{\pa D}\frac{\pa\Phi_0(x,y)}{\pa\nu(y)}\Phi_0(y,x_j)ds(y)~~~&{\rm on}~\pa D, 
  \end{array}
  \right.
  \enn
  where $h_j(x)=\Phi_0(x,x_j)$. Again by \cite[Corollary 3.7]{DR13}, we know that the trace of $\wid f_j+h_j$ on $\pa D$ is uniformly bounded in $H^{1/2}(\pa D)$. Thus the standard elliptic regularity gives $\|\wid f_j+h_j\|_{H^1(D)}\leq C$ uniformly for $j\in\N$. Then we have 
  \be\label{2.3a}
    \left\|\hat v_j-\frac{2}{\lambda(x_0)+1}\Phi_0(x,x_j)\right\|_{H^{1}(D\setminus\ov D_b)}\leq C 
  \en
  for problem \eqref{2.12} with $k^2/\lambda_0=k_1^2>0$. 


  Finally we focus on the difference between $v_j$ and $\hat v_j$. Define $(\wid u_j^s,\wid v_j):=(u_j-\Phi(\cdot,y_j),v_j)-(\hat u_j^s,\hat v_j)$, it then follows that 
  \be\label{2.13}
  \left\{
  \begin{array}{ll}
    \Delta\wid u^s_j+k^2\wid u^s_j=0~~~&{\rm in}~\R^3\setminus\ov D,\\
    \dive(\lambda\nabla\wid v_j)+k^2n\wid v_j=\dive((\lambda(x_0)-\lambda)\na\hat v_j)+k^2(1-n)\hat v_j~~~&{\rm in}~D\setminus\overline{D_b},\\
    \wid u^s_j-\wid v_j=0~~~&{\rm on}~\pa D,\\
    \pa_\nu\wid u^s_j-\left(\lambda\pa_\nu\wid v_j+\gamma\wid v_j\right)=(\lambda-\lambda(x_0))\pa_\nu\hat v_j~~~&{\rm on}~\pa D,\\
    \mathscr{B}(\wid v_j)=0~~~&{\rm on}~\pa D_b,\\
    \displaystyle\lim\limits_{r\rightarrow\infty}r\left(\pa_r\wid u^s_j-ik\wid u^s_j\right)=0,~&r=|x|.
  \end{array}
  \right.
  \en
Since 
 \ben
   \hat v_j=\left(\hat v_j-\frac{2}{\lambda(x_0)+1}\Phi_0(x,x_j)\right)+\frac{2}{\lambda(x_0)+1}\Phi_0(x,x_j), 
 \enn
 we obtain from \eqref{2.3a} that 
 \ben
   \|\dive((\lambda(x_0)-\lambda)\na\hat v_j)+k^2(1-n)\hat v_j\|_{H^{-1}(D\setminus\ov D_b)}+\|(\lambda-\lambda(x_0))\pa_\nu\hat v_j\|_{H^{-1/2}(\pa D)}\leq C
 \enn
 uniformly for $j\in\N$. Hence, by the variational method we derive that 
  \be\label{2.4a}
    \|\wid u_j^s\|_{H^1(B_R\setminus\ov D)}+\|\wid v_j\|_{H^1(D\setminus\ov D_b)}\leq C
  \en
  The proof is then finished by combining \eqref{2.3a} and \eqref{2.4a}. 
\end{proof}

Moreover, if $\lambda=1$ near the boundary $\pa D$, we need to consider the hypersingular point source incidence. To this end, we refer the following theorem (see \cite[Theorem 2.7]{JBH18}). 
\begin{theorem}\label{thm2.4}
  With $u^i(x)=\nabla_x\Phi(x,z)\cdot \vec a$, $\lambda\equiv1$ and $\gamma\equiv0$ for problem \eqref{1.1}, where $z\in\R^3\setminus\ov D$ and $\vec a\in\R^3$ is fixed, the unique solution $(u^s,v)$ satisfies 
  \ben
    \|v\|_{L^p(D\setminus\ov D_b)}+\|v-u^i\|_{H^1(D\setminus\ov D_b)}\leq C(\|u^i\|_{L^p(D\setminus\ov D_b)}+\|\mathscr B(u^i)\|_{L^\infty(\pa D_b)}), 
  \enn
  where $6/5\leq p<3/2$. 
\end{theorem}

Based on Theorems \ref{thm2.8} and \ref{thm2.4}, we can further prove the following result for the case $\gamma\neq0$. 
\begin{theorem}\label{thm2.7}
	Suppose $\gamma\in C^{0,\alpha_1}\left(\pa D\right)$ with $1/2<\alpha_1\leq1$ and there exists an open neighborhood $O_1(\pa D)$ of $\pa D$ such that $\lambda=1$ in $O_1(\pa D)$. Let $(u^s_j,v_j)$ be the unique solution for problem {\rm (\ref{1.1})} with the incident wave $u^i_j(x)=\na_x\Phi(x,x_j)\cdot\nu(x_0)$. Then we have 
	$$\left\|v_j-u_j^i+(\gamma(x_0)/2)\Phi_0(x,x_j)\right\|_{H^1(D\setminus\ov D_b)}\leq C$$
	uniformly for $j\in\N$.
\end{theorem}
\begin{proof}
  By Theorem \ref{thm2.4}, there exists a unique solution $(\hat u^s_j,\hat v_j)$ for problem \eqref{1.1} with $\lambda\equiv1, \gamma\equiv0$ and incident waves $u^i_j$ 
  such that
  \be\label{2.7} 
    \|\hat v_j-u^i_j\|_{H^1(D\setminus\ov D_b)}\leq C(\|u^i_j\|_{L^p(D\setminus\ov D_b)}+\|\mathscr B(u^i_j)\|_{L^\infty(\pa D_b)})\leq C. 
  \en
  From Theorem \ref{thm2.8}, the unique solution $(\hat U^s_j,\hat V_j)$ for problem \eqref{1.1} with $\lambda\equiv1$ and the incident waves $-\gamma(x_0)\Phi(\cdot,x_j)$ 
  satisfies 
  \be\label{2.9}
    \|\hat V_j+\gamma(x_0)\Phi(x,x_j)\|_{H^1(D\setminus\ov D_b)}\leq C. 
  \en

  Now consider the following problem
  \be\label{2.10}
  \left\{
  \begin{array}{ll}
    \Delta\wid u^s_j+k^2\wid u^s_j=0~~~&{\rm in}~\R^3\setminus\ov D,\\
    \Delta\wid v_j+k^2n\wid v_j=0~~~&{\rm in}~D\setminus\overline{D_b},\\
    \wid u^s_j-\wid v_j=-\gamma(x_0)\Phi(x,x_j)~~~&{\rm on}~\pa D,\\
    \pa_\nu\wid u^s_j-\left(\pa_\nu\wid v_j+\gamma\wid v_j\right)=\gamma(\hat v_j+\hat V_j)-\gamma(x_0)\pa_\nu \Phi(x,x_j)~~~&{\rm on}~\pa D,\\
    \mathscr{B}(\wid v_j)=0~~~&{\rm on}~\pa D_b,\\
    \displaystyle\lim\limits_{r\rightarrow\infty}r\left(\pa_r\wid u^s_j-ik\wid u^s_j\right)=0,~&r=|x|.
  \end{array}
  \right.
  \en
  Clearly, $\Phi(x,x_j)\in L^p(\pa D)$ for $1\leq p<2$ and $j\in\N$. Further, it can be verified that
  \be\label{2.11}\nonumber
    \gamma(\hat v_j+\hat V_j)-\gamma(x_0)\pa_\nu \Phi(x,x_j)=&&\;\gamma(\hat v_j-u^i_j)+\gamma(\hat V_j+\gamma(x_0)\Phi(x,x_j))\\ \nonumber
    &&\;-\gamma(x_0)\gamma\Phi(x,x_j)+(\gamma(x)-\gamma(x_0))u^i_j\\
    &&\;+\gamma(x_0)\na_x\Phi(x,x_j)\cdot(\nu(x_0)-\nu(x)),
  \en
  which implies that the right hand side of (\ref{2.11}) is uniformly bounded in $L^p(\pa D)$ with $4/3\leq p<2/(2-\alpha_1)$ due to (\ref{2.7}), (\ref{2.9}) as well as the regularity of the boundary $\pa D$ and $\gamma$. Therefore, arguing similarly as in the proof of Theorem \ref{thm2.8}, we can obtain that
  \be\label{2.114}
  \|\wid v_j-(\gamma(x_0)/2)\Phi_0(x,x_j)\|_{H^1(D\setminus\ov D_b)}\leq C.
  \en

  Let $U^s_j=\hat u^s_j+\hat U^s_j+\wid u^s_j$ and $V_j=\hat v_j+\hat V_j+\wid v_j$. Finally, we define $(\wid U_j^s,\wid V_j)$ by
  \be\label{2.115}
  \left\{
  \begin{array}{ll}
    \Delta\wid U^s_j+k^2\wid U^s_j=0~~~&{\rm in}~\R^3\setminus\ov D,\\
    \dive(\lambda\na\wid V_j)+k^2n\wid V_j=\dive((1-\lambda)\nabla V_j)~~~&{\rm in}~D\setminus\overline{D_b},\\
    \wid U^s_j-\wid V_j=0~~~&{\rm on}~\pa D,\\
    \pa_\nu\wid U^s_j-\left(\pa_\nu\wid V_j+\gamma\wid V_j\right)=0~~~&{\rm on}~\pa D,\\
    \mathscr{B}(\wid V_j)=0~~~&{\rm on}~\pa D_b,\\
    \displaystyle\lim\limits_{r\rightarrow\infty}r\left(\pa_r\wid U^s_j-ik\wid U^s_j\right)=0,~&r=|x|.
  \end{array}
  \right.
  \en
  It is easily checked that $u_j^s=U_j^s+\wid U_j^s$ and $v_j=V_j+\wid V_j$. Since $\lambda=1$ in $O_1(\pa D)$, it follows that $\|\dive((1-\lambda)\nabla V_j)\|_{H^{-1}(D\setminus\ov D_b)}\leq C$, which implies by the variational method that $(\wid U_j^s,\wid V_j)$ is well defined and $\|\wid V_j\|_{H^1(D\setminus\ov D_b)}\leq C$ uniformly for $j\in\N$. Therefore, combining (\ref{2.7}), (\ref{2.9}) and (\ref{2.114}), we obtain the desired estimate and the proof is complete. 
\end{proof}
\begin{remark}\label{rem2.3}
	The estimates in Theorems {\rm \ref{thm2.8}} and {\rm \ref{thm2.7}} are of great significance in the uniqueness proof of the inverse problem. We hope that these kinds estimates can be extended to other cases and genearlized to derive a better uniqueness results for the inverse problem. 
\end{remark}

\section{The inverse problem}\label{sec3}
In this section, we shall study two interior transmission problems and then apply a novel technique to prove the uniqueness theorem in determination of the object and relating coefficients. We show that the object $D$, the transmission coefficient $\lambda$, the conductive coefficient $\gamma$, the refractive index $n$ and also some of their derivatives can be uniquely recovered, for which the proof is greatly relied on the singularity of the solutions and the aforementioned interior transmission problems.

\subsection{The interior transmission problem}\label{sec3.1}
\setcounter{equation}{0}
In this subsection, we introduce two types of interior transmission problems corresponding to the case $\lambda\neq1$ and $\lambda=1$ near the boundary $\pa D$, respectively. 

Denote by $\Om$ a bounded and simply connected domain in $\R^3$ with $\pa\Om\in C^2$. For the case $\lambda\neq1$, we consider the following interior transmission problem
\be\label{3.1}
  \left\{
  \begin{array}{ll}
    \dive(\lambda_1\na U)-b_1U=g_1~~~&{\rm in}~\Om,\\
    \dive(\lambda_2\na V)-b_2V=g_2~~~&{\rm in}~\Om,\\
    U-V=h_1~~~&{\rm on}~\pa\Om,\\
    \lambda_1\pa_\nu U-\lambda_2\pa_\nu V=h_2~~~&{\rm on}~\pa\Om,
  \end{array}
  \right.
\en
with $\lambda_1,\lambda_2\in C(\ov\Om)$, $g_1,g_2\in L^2(\Om)$, $h_1\in H^{1/2}(\pa\Om)$, $h_2\in H^{-1/2}(\pa\Om)$ and $b_1$, $b_2$ positive constants. This problem was studied in \cite{JY21}, and we have the result below (see \cite[Corollary 2.2]{JY21}).
\begin{lem}\label{lem3.1}
  Assume $\lambda_i>0$ in $\ov\Om$ with $i=1,2$. If there exists a positive constant $\varepsilon<\min\{b_2,1\}$ such that
  \ben
    \varepsilon b_1>\left(\frac{1+b_2}{2}\right)^2~~~{\rm and}~~~\varepsilon\left(\inf\limits_{\Om}\frac{\lambda_1}{\lambda_2}\right)>\left(\frac{1+b_2}{2}\right)^2,
  \enn
  then the problem {\rm (\ref{3.1})} has a unique solution $(U,V)$ satisfying that
  \ben
    \|U\|_{H^1(\Om)}+\|V\|_{H^1(\Om)}\leq C(\|g_1\|_{L^2(\Om)}+\| g_2\|_{L^2(\Om)}+\|h_1\|_{H^{1/2}(\pa\Om)}+\|h_2\|_{H^{-1/2}(\pa\Om)}).
  \enn
\end{lem}
\begin{remark}\label{remark3.2}
  The condition $\varepsilon<b_2$ excludes the case that $\inf\limits_{\Om}(\lambda_1/\lambda_2)\leq1$ since $[(1+b_2)/2]^2\geq b_2$. But sometimes we may exchange $\lambda_1$ and $\lambda_2$ to make the conditions can be satisfied.
\end{remark}

For the case $\lambda=1$, we introduce the interior transmission problem of the following type
\be\label{3.2}
  \left\{
  \begin{array}{ll}
    \Delta v+k^2n_1v=0~~~&{\rm in}~\Om,\\
    \Delta w+k^2n_2w=0~~~&{\rm in}~\Om,\\
    v-w=f_1~~~&{\rm on}~\pa\Om,\\
    \pa_\nu v+\eta w-\pa_\nu w=f_2~~~&{\rm on}~\pa\Om,
  \end{array}
  \right.
\en
where $f_1\in H^{1/2}(\pa\Om)$, $f_2\in H^{-1/2}(\pa\Om)$, $n_1,n_2\in L^\infty(\Om)$ and $\eta\in L^\infty(\pa\Om)$. To get the well-posedness of the problem (\ref{3.2}), we give a lemma first.
\begin{lem}\label{lem3.2}
  Let $\Om\subset\R^n(n\geq2)$ be a bounded domain with $\pa\Om\in C^2$. Define 
  $$\displaystyle C_1(\Om):=\ds\mathop{\sup\limits_{u\in \wid H_0^2(\Om)}}\limits_{u\neq0}\frac{\displaystyle\int_{\pa\Om}|\pa_\nu u|^2ds}{\displaystyle\int_{\Om}|\Delta u|^2dx},$$
  where $\wid H^2_0(\Om)=H^2(\Om)\cap H^1_0(\Om)$. Then $C_1(\Om)\rightarrow0$ as $diam(\Om)\rightarrow0$.
\end{lem}
\begin{proof}
  The standard elliptic regularity gives 
  \ben
    \|u\|_{H^2(\Om)}\leq C\|\Delta u\|_{L^2(\Om)},~~~\forall u\in \wid H^2_0(\Om).
  \enn
  Then by the trace theorem we obtain that $C_1(\Om)<\infty$ for any bounded domain $\Om$ with $\pa\Om\in C^2$.

  Without loss of generality, we only consider the case that $\Om$ are open balls in $\R^n$. Set $C_0=C_1(B_1)$, in other words,
  \ben
    \int_{\pa B_1}\left|\pa_\nu u\right|^2ds\leq C_0\int_{B_1}|\Delta u|^2dx,~~~\forall u\in \wid H^2_0(B_1).
  \enn
  Now for fixed $B_R(x_0)\subset\R^n$, for any $u\in\wid H^2_0(B_R(x_0))$, let $v(y)=u(x_0+Ry)$, $y\in B_1$. Then $v(y)\in\wid H^2_0(B_1)$. Hence, we have
  \ben
    \int_{\pa B_1}\left|\pa_\nu v\right|^2ds\leq C_0\int_{B_1}|\Delta v|^2dx.
  \enn
  Changing the integral variable gives that
  \ben
    &&\int_{\pa B_1}\left|\pa_\nu v\right|^2ds=R^{3-n}\int_{\pa B_R(x_0)}\left|\pa_\nu u\right|^2ds,\\
    &&\int_{B_1}|\Delta v|^2dx=R^{4-n}\int_{B_R(x_0)}|\Delta u|^2dx,
  \enn
  which implies that
  \ben
    \int_{\pa B_R(x_0)}\left|\pa_\nu u\right|^2ds\leq C_0R\int_{B_R(x_0)}|\Delta u|^2dx,~~~\forall u\in \wid H^2_0(B_R(x_0)).
  \enn
  Thus we obtain $C_1(B_R(x_0))\leq C_0R$ and the conclusion follows.
\end{proof}

Denote by $\lambda_1(\Om)$ the first Dirichlet eigenvalue of $-\Delta$ in $\Om$. We recall that 
$$\lambda_1(\Om)=\mathop{\inf\limits_{u\in H^1_0(\Om)}}\limits_{u\neq0}\frac{\displaystyle\int_{\Om}|\nabla u|^2dx}{\displaystyle\int_{\Om}|u|^2dx}$$ 
and $\lambda_1(\Om)\rightarrow\infty$ as $diam(\Om)\rightarrow0$.
Define the Hilbert space
\ben
  H_\Delta^1(\Om):=\{u\in H^1(\Om),\Delta u\in L^2(\Om)\}
\enn
with the inner product $(u,v)_{H_\Delta^1(\Om)}=(u,v)_{H^1(\Om)}+(\Delta u,\Delta v)_{L^2(\Om)}$. Clearly, for $u\in H_\Delta^1(\Om)$ we have that $u|_{\pa\Om}\in H^{1/2}(\pa\Om)$ and $\partial_\nu u\in H^{-1/2}(\pa\Om)$. In particular, if $u=\pa_\nu u=0$ on $\pa\Om$ for some $u\in H^1_\Delta(\Om)$, then $u\in H^2_0(\Om)$. Now we give the following uniqueness and existence theorem for problem (\ref{3.2}). 
\begin{theorem}\label{thm3.3}
  Suppose there exists a $u_0\in H_\Delta^1(\Om)$ such that $u_0=f_1$ and $\pa_\nu u_0=f_2$ on $\pa\Om$. Assume in addition that $n_1,n_2>\delta_1$ with $|n_1-n_2|>\delta_1$ for some constant $\delta_1>0$. Then problem {\rm (\ref{3.2})} has a unique solution $(v,w)\in L^2(\Om)\times L^2(\Om)$ such that
  \ben
    \|v\|_{L^2(\Om)}+\|w\|_{L^2(\Om)}\leq C\|u_0\|_{H_\Delta^1(\Om)},
  \enn
  provided $(k,n_1,n_2,\eta,\Om)$ satisfying one of the following conditions
  \begin{enumerate}[{\rm 1.}]
    \item $n_1-n_2<0, \eta>\delta_2$ and $\displaystyle k^2<\frac{\lambda_1(\Om)\inf_\Om n_2}{(\sup_{\Om} n_2)^2}$, \\
    \item $n_1-n_2>0, \eta<-\delta_2$ and $\displaystyle k^2<\frac{\lambda_1(\Om)\inf_\Om n_1}{(\sup_\Om n_1)^2}$, \\
    \item $n_1-n_2>0, \eta>\delta_2, $and $$0<\frac{C'C_1(\Om)k^2}{C'\inf_\Om|\eta|-C_1(\Om)k^2}<\frac{\lambda_1(\Om)}{k^2(\sup_\Om n_1)^2}-\frac{1}{\inf_\Om n_1},$$
    \item $n_1-n_2<0, \eta<-\delta_2, $and $$0<\frac{C'C_1(\Om)k^2}{C'\inf_\Om|\eta|-C_1(\Om)k^2}<\frac{\lambda_1(\Om)}{k^2(\sup_\Om n_2)^2}-\frac{1}{\inf_\Om n_2},$$ 
    \item $n_1-n_2<0, \eta=0$ and $\displaystyle k^2<\frac{\lambda_1(\Om)\inf_\Om n_2}{(\sup_\Om n_2)^2}$, \\
    \item $n_1-n_2>0, \eta=0$ and $\displaystyle k^2<\frac{\lambda_1(\Om)\inf_\Om n_1}{(\sup_\Om n_1)^2}$, \\
  \end{enumerate}
  for some positive constant $\delta_2$ and $C':=1/\sup_\Om|n_1-n_2|$.
\end{theorem}
\begin{proof}
  First we consider the case $\eta\neq0$ corresponding to conditions 1-4. Set $u=v-w$, then it can be verified that 
  \be\label{3.3}
  \left\{
  \begin{array}{ll}
    \displaystyle(\Delta+k^2n_2)\frac{1}{n_2-n_1}(\Delta+k^2n_1)u=0~~~&{\rm in}~\Om,\\
    u=f_1~~~&{\rm on}~\pa\Om,\\
    \displaystyle\pa_\nu u=f_2-\frac{\eta}{k^2(n_2-n_1)}(\Delta u+k^2n_1u)~~~&{\rm on}~\pa\Om.
  \end{array}
  \right.
  \en
  Define $u'=u-u_0\in\wid H^2_0(\Om)$. By Green's formula we get that the variational formulation of (\ref{3.3}) is to find a $u'\in\wid H^2_0(\Om)$ such that
  \ben
    \displaystyle B(u',\varphi)=F(\varphi)~~~\forall\varphi\in\wid H^2_0(\Om),
  \enn
  where
  \ben
    &&B(u',\varphi):=\displaystyle\int_{\Om}\frac{1}{n_2-n_1}(\Delta u'+k^2n_1u')(\Delta \ov\varphi+k^2n_2\ov\varphi)dx+\int_{\pa\Om}\frac{k^2}{\eta}\pa_\nu u'\pa_\nu\ov\varphi ds, \\
    &&F(\varphi):=\int_{\Om}\frac{1}{n_1-n_2}(\Delta u_0+k^2n_1u_0)(\Delta \ov\varphi+k^2n_2\ov\varphi)dx.
  \enn
  Clearly, by the H\"{o}lder's inequality and the trace theorem, we have that 
  $B(u',\varphi)$ and $F(\varphi)$ are both bounded on $\wid H_0^2(\Om)$. 

   Note that $\|u'\|_{H^2(\Om)}\leq C\|\Delta u'\|_{L^2(\Om)}$ for all $u'\in\wid H^2_0(\Om)$. Now for $(k,n_1,n_2,\eta,\Om)$ satisfying condition 1, we deduce the following
  \ben
    B(u',u')=&&\;\int_{\Om}\frac{1}{n_2-n_1}|\Delta u'+k^2n_2u'|^2dx+\int_{\pa\Om}\frac{k^2}{\eta}\left|\pa_\nu u'\right|^2ds+k^2\int_{\Om}|\nabla u'|^2dx\\ 
    &&\;-k^4\int_{\Om}n_2|u'|^2dx\\
    \geq&&\;\left(C'-\frac{C^{'2}}{\varepsilon_1}\right)\|\Delta u'\|_{L^2(\Om)}^2+\varepsilon_1\left(\frac{C'}{\varepsilon_1}\|\Delta u'\|_{L^2(\Om)}-k^2\|n_2u'\|_{L^2(\Om)}\right)^2\\ 
    &&\;+k^4\left(C'+\frac{\lambda_1(\Om)}{k^2(\sup_\Om n_2)^2}-\frac{1}{\inf_\Om n_2}-\varepsilon_1\right)\|n_2u'\|_{L^2(\Om)}^2 \\
    \geq&&\;C\|\Delta u'\|_{L^2(\Om)}^2
  \enn
  for all $u'\in\wid H^2_0(\Om)$ with the constant $\varepsilon_1$ satisfying that $0<\varepsilon_1-C'<\lambda_1(\Om)/[k^2(\sup_\Om n_2)^2]-1/\inf_\Om n_2$. Thus we obtain that $B(u',\varphi)$ is coercive on $\wid H^2_0(\Om)$ when $(k,n_1,n_2,\eta,\Om)$ satisfies condition 1. Similar calculation also yields that $B(u',\varphi)$ is coercive with condition 2.

  Under condition 3, we choose a positive constant $\varepsilon_2$ such that $C'C_1(\Om)k^2/(C'\inf_\Om|\eta|-C_1(\Om)k^2)<\varepsilon_2<\lambda_1(\Om)/[k^2(\sup_\Om n_1)^2]-1/\inf_\Om n_1$. Then for all $u'\in\wid H^2_0(\Om)$ we derive 
  \ben
    -B(u',u')=&&\;\int_{\Om}\frac{1}{n_1-n_2}|\Delta u'+k^2n_1u'|^2dx+k^2\int_{\Om}|\nabla u'|^2dx-k^4\int_{\Om}n_1|u'|^2dx\\ 
    &&\;-\int_{\pa\Om}\frac{k^2}{\eta}\left|\pa_\nu u'\right|^2ds\\ 
    \geq&&\;\left(C'-\frac{k^2C_1(\Om)}{\inf_\Om|\eta|}-\frac{C^{'2}}{\varepsilon_2+C'}\right)\|\Delta u'\|_{L^2(\Om)}^2\\
    &&\;+(\varepsilon_2+C')\left(\frac{C'}{\varepsilon_2+C'}\|\Delta u'\|_{L^2(\Om)}-k^2\|n_1u'\|_{L^2(\Om)}\right)^2\\ 
    &&\;+k^4\left(\frac{\lambda_1(\Om)}{k^2(\sup_\Om n_1)^2}-\frac{1}{\inf_\Om n_1}-\varepsilon_2\right)\|n_1u'\|_{L^2(\Om)}^2 \\
    \geq&&\;C\|\Delta u'\|_{L^2(\Om)}^2. 
  \enn
  The coercivity of $B(u',\varphi)$ under condition 4 follows analogously as above.

  Then we consider the case $\eta=0$, which corresponds to conditions 5 and 6. We see that $u=v-w$ satisfies
  \be\label{3.4}
  \left\{
  \begin{array}{ll}
    \displaystyle(\Delta+k^2n_2)\frac{1}{n_2-n_1}(\Delta+k^2n_1)u=0~~~&{\rm in}~\Om,\\[1mm]
    u=f_1~~~&{\rm on}~\pa\Om,\\[1mm]
    \pa_\nu u=f_2~~~&{\rm on}~\pa\Om.
  \end{array}
  \right.
  \en
  Set $u'=u-u_0\in H_0^2(\Om)$. Also we can get the variational problem of (\ref{3.4}): finding $u'\in H_0^2(\Om)$ such that
  \ben
    \wid B(u',\varphi)=F(\varphi)~~~\forall\varphi\in H^2_0(\Om),
  \enn
  where
  \ben
    \wid B(u',\varphi):=\displaystyle\int_{\Om}\frac{1}{n_2-n_1}(\Delta u'+k^2n_1u')(\Delta \ov\varphi+k^2n_2\ov\varphi)dx.
  \enn
  Obviously, $\wid B(u',\varphi)$ is still bounded on $H_0^2(\Om)$. Following the previous calculation closely, it can be shown that $\wid B(u',\varphi)$ is coercive under conditions 5 and 6.

  Now by the Lax-Milgram theorem, with conditions 1-6, there exists a unique solution $u'\in\wid H_0^2(\Om)$ or $H_0^2(\Om)$ such that
  \ben
    \|u'\|_{H^2(\Om)}\leq C\|u_0\|_{H^1_\Delta(\Om)}.
  \enn
  Define $v=[1/k^2(n_2-n_1)](\Delta u+k^2n_2u)$ and $w=v-u$. Then it can be verified that $(v,w)\in L^2(\Om)\times L^2(\Om)$ is the unique solution to problem (\ref{3.2}) and the desired a priori estimate holds. 
\end{proof}
\begin{remark}\label{remark3.4}
  Since $\lambda_1(\Om)\rightarrow\infty$ and $C_1(\Om)\rightarrow0$ as $diam(\Om)\rightarrow0$, we see that conditions {\rm 1-6} will hold for fixed $k,n_1,n_2$ and $\eta$ with $\Om$ sufficiently small. In other words, problem {\rm (3.2)} is always well-posed in small domains. 
\end{remark}

\subsection{Uniqueness of the inverse problem}\label{sec3.2}
\setcounter{equation}{0}
  With all the preceding preparations, we are ready to prove the uniqueness theorem of the inverse problem. To this end, we first give an assumption for some functions $f$ in $D$.
\begin{assumption}\label{asm4.1}
  There exists an open neighborhood $O_2(\pa D)$ of $\pa D$, $O_2(\pa D)\subset\subset D\setminus\ov D_b$, and a constant $\varepsilon>0$ such that $f(x)>1+\varepsilon$ or $\varepsilon<f(x)<1-\varepsilon$ for a.e. $x\in O_2(\pa D)$.
\end{assumption}
\begin{theorem}\label{thm4.2}
  Suppose $(D_1,\lambda_1,\gamma_1,n_1,D_b^1)$ and $(D_2,\lambda_2,\gamma_2,n_2,D_b^2)$ are two scatters. Denote by $u_\infty^m(\hat x;d)$ the far-field pattern of the scattered field $u^s_m(x;d)$ to transmission problem {\rm (\ref{1.1})} corresponding to the scatter $(D_m,D_b^m,n_m,\lambda_m,\gamma_m)$ with the incident wave $u^i(x)=e^{ikx\cdot d}$, $d\in\Sp^2$, $m=1,2$. Assume that $u_\infty^1(\hat x;d)=u_\infty^2(\hat x;d)$ for all $\hat x,d\in\Sp^2$. \\
  {\rm (i)}If Assumption {\rm \ref{asm4.1}} holds for $\lambda_m$ in $D_m$ with $m=1,2$, then $D_1=D_2=D$ and $\lambda_1=\lambda_2$ on $\pa D$. Suppose in addition that $\lambda_m\in C^2(\ov D_m\setminus D_b^m)$ and $\gamma_m\in C(\pa D_m)$, $m=1,2$, then $(\pa_\nu\lambda_1-\pa_\nu\lambda_2)+2(\gamma_2-\gamma_1)=0$ on $\pa D$. Define
  $$a_m=\frac{\Delta\sqrt{\lambda_m}}{\sqrt{\lambda_m}}-\frac{k^2n_m}{\lambda_m},~~m=1,2. $$
  If further $a_m\in C^1(\ov D_m\setminus D_b^m)$, $m=1,2$, then $a_1=a_2$ on $\pa D$.
  \\
  {\rm (ii)}If $\lambda_m=1$ in $O_1(\pa D_m)$, $\gamma_m\in C^{0,\alpha_1}(\pa D_m)$ is real-valued and $n_m$ satisfies Assumption {\rm \ref{asm4.1}} in $D_m$ with $m=1,2$, then $D_1=D_2=D$ and $\gamma_1=\gamma_2$ on $\pa D$. Further suppose that $n_m\in C^2(\ov D_m\setminus D_b^m)$ is real-valued, $m=1,2$, then $D^\alpha n_1=D^\alpha n_2$ on $\pa D$ for $|\alpha|\leq2$.
\end{theorem}
\begin{proof}
  (i)Suppose $D_1\neq D_2$. Denote by $G$ the unbounded connected part of $\R^3\setminus\overline{(D_1\cup D_2)}$. Then we can find a point $x_0\in\R^3$ and a small ball $B$ centered at $x_0$ such that $x_0\in\pa G\cap\pa D_1$, $x_0\notin\pa D_2$ and $B\cap (\overline{D_2\cup D_b^1})=\emptyset$. Define
  \ben
    x_j:=x_0+\frac{\delta}{j}n(x_0),~~~j=1,2,\cdots
  \enn
  with $\delta>0$ small enough such that $x_j\in B$ for all $j\in\N$. Since $\pa D_1\in C^{2,\alpha_0}$, we can further find a small $C^{2,\alpha_0}$ domain $D_0$ such that $B\cap D_1\subset D_0\subset (D_1\setminus(\overline{D_2\cup D_b^1}))\cap O_2(\pa D_1)$.

  Denote by $(u^s_{m}(x;y),v_{m}(x;y))$ the unique solution of problem (\ref{1.1}) corresponding to the scatterer $(D_m,\lambda_m,\gamma_m,n_m,D_b^m)$ and the incident wave $u^i=\Phi(x,y)$ with $m=1,2$ and $y\in G$. Let $u_\infty^m(\hat x;y)$ be the far-field pattern of $u_m^s(x;y)$. Then by Rellich's Lemma, we have $u^s_1(x;d)=u^s_2(x;d)$, $\forall x\in\ov G$, $d\in\Sp^2$. Arguing similarly as in the proof of \cite[Theorem 3.16]{DR13}, we can obtain the mixed reciprocity relation $4\pi u_\infty^m(-d; x)=u^s_m( x;d)$, $\forall x\in\ov G$, $d\in\Sp^2$ and $m=1,2$. Thus we further have $u_\infty^1(\hat x;x_j)=u_\infty^2(\hat x;x_j)$, $\forall\hat x\in\Sp^2$ and $j\in\N$, which implies that $u^s_1(x;x_j)=u^s_2(x;x_j)$, $\forall x\in\ov G$ and $j\in\N$.

  Hereafter, we set $u_{m,j}^s=u_m^s(x;x_j)$, $v_{m,j}=v_m(x;x_j)$ and $u_{m,j}=\Phi(\cdot,x_j)+u_{m,j}^s$ with $m=1,2$. It can be verified that $(v_{1,j},u_{2,j})$ solves the following interior transmission problem
  \be\label{4.1}
  \left\{
  \begin{array}{ll}
    \dive(\lambda_1\na v_{1,j})-b_1v_{1,j}=g_{1,j}~~~&{\rm in}~D_0,\\
    \Delta u_{2,j}-b_2u_{2,j}=g_{2,j}~~~&{\rm in}~D_0,\\
    v_{1,j}-u_{2,j}=h_{1,j}~~~&{\rm on}~\pa D_0,\\
    \lambda_1\pa_\nu v_{1,j}-\pa_\nu u_{2,j}=h_{2,j}~~~&{\rm on}~\pa D_0,
  \end{array}
  \right.
  \en
  where
  \ben
    &&g_{1,j}=-(k^2n_1+b_1)v_{1,j},\quad g_{2,j}=-(k^2+b_2)u_{2,j}, \\
    &&h_{1,j}=(v_{1,j}-u_{2,j})|_{\pa D_0},\quad\; h_{2,j}=\left.\left(\displaystyle\lambda_1\pa_\nu v_{1,j}-\pa_\nu u_{2,j}\right)\right|_{\pa D_0},
  \enn
  and $b_1$, $b_2$ are positive constants to be chosen such that the conditions in Lemma \ref{lem3.1} are satisfied. We note that such $b_1,b_2$ exist because Assumption \ref{asm4.1} holds for $\lambda_1$. Clearly, $h_{1,j}=0$ and $h_{2,j}=-\gamma_1u_{2,j}$ on $\Gamma:=\pa D_0\cap\pa D_1$. By Theorem \ref{thm2.8}, we have that $g_{1,j}$ are uniformly bounded in $L^2(D_0)$ for $j\in\N$. Due to the positive distance between $x_j$ and $D_2$, it follows that $\|u_{2,j}^s\|_{H^1(D_0)}\leq C$ from the well-posedness of problem (\ref{1.1}), which indicates that $g_{2,j}\in L^2(D_0)$ are uniformly bounded for $j\in\N$ since $u^i_j=\Phi(x,x_j)\in L^2(D_0)$ are uniformly bounded.

  Next we show that $h_{1,j}$ and $h_{2,j}$ are uniformly bounded in $H^{1/2}(\pa D_0)$ and $H^{-1/2}(\pa D_0)$ for $j\in\N$, respectively. By Theorem \ref{thm2.8}, we see that
  \ben
    \|h_{1,j}\|_{H^{1/2}(\pa D_0)}\;&&=\|v_{1,j}-u_{2,j}\|_{H^{1/2}(\pa D_0\setminus\Gamma)}\\
    &&\leq C(\|v_{1,j}\|_{H^1(D_0\setminus\ov B)}+\|u_{2,j}^s\|_{H^1(D_0)}+\|\Phi(\cdot,x_j)\|_{H^1(D_0\setminus\ov B)})\\
    &&\leq C
  \enn
  uniformly for $j\in\N$. Further, it is derived that 
  \ben
    \|h_{2,j}\|_{H^{-1/2}(\pa D_0)}\;&&\leq \left(\|h_{2,j}\|_{H^{-1/2}(\pa D_0\setminus\Gamma)}+\|\gamma_1 u_{2,j}\|_{H^{-1/2}(\Gamma)}\right)\\
    &&\leq C\left(\|v_{1,j}\|_{H^1(D_0\setminus\ov B)}+\|u_{2,j}^s\|_{H^1(D_0)}+\|\Phi(\cdot,x_j)\|_{H^1(D_0\setminus\ov B)}\right.\\
    &&\quad\qquad\left.+\|u_{2,j}\|_{H^{-1/2}(\Gamma)}\right)\\
    &&\leq C\left(\|v_{1,j}\|_{H^1(D_0\setminus\ov B)}+\|u_{2,j}^s\|_{H^1(D_0)}+\|\Phi(\cdot,x_j)\|_{H^1(D_0\setminus\ov B)}\right.\\
    &&\quad\qquad\left.+\|\Phi(\cdot,x_j)\|_{L^p(\Gamma)}\right)\\
    &&\leq C
  \enn
  with $4/3<p<2$ due to the fact that $L^p(\Gamma)$ is bounded embedded into $H^{-1/2}(\Gamma)$ for $p>4/3$ and $\Phi(\cdot,x_j)$ are uniformly bounded in $L^p(\Gamma)$ for $p<2$.

  Therefore, we obtain that $\|u_{2,j}\|_{H^1(D_0)}\leq C$ by Lemma \ref{lem3.1}, which is however a contradiction since  $\|u_{2,j}^s\|_{H^1(D_0)}\leq C$ and $\|\Phi(x,x_j)\|_{H^1(D_0)}\rightarrow\infty$ as $j\rightarrow\infty$. Thus $D_1=D_2=D$.

  Now we show $\lambda_1=\lambda_2$ on $\pa D$. Suppose not, then we can assume that $\lambda_1/\lambda_2>1$ in $\ov D_0$. It is deduced that  $(v_{1,j},v_{2,j})$ is the solution to the following problem
  \be\label{4.3}
  \left\{
  \begin{array}{ll}
    \dive(\lambda_1\na v_{1,j})-b_1v_{1,j}=\wid g_{1,j}~~~&{\rm in}~D_0,\\
    \dive(\lambda_2\na v_{2,j})-b_2v_{2,j}=\wid g_{2,j}~~~&{\rm in}~D_0,\\
    v_{1,j}-v_{2,j}=\wid h_{1,j}~~~&{\rm on}~\pa D_0,\\
    \lambda_1\pa_\nu v_{1,j}-\lambda_2\pa_\nu v_{2,j}=\wid h_{2,j}~~~&{\rm on}~\pa D_0,
  \end{array}
  \right.
  \en
  where
  \ben
    &&\wid g_{1,j}=-(k^2n_1+b_1)v_{1,j},\quad \wid g_{2,j}=-(k^2n_2+b_2)v_{2,j}, \\
    &&\wid h_{1,j}=(v_{1,j}-v_{2,j})|_{\pa D_0},\quad\;\wid h_{2,j}=\left.\left(\displaystyle\lambda_1\pa_\nu v_{1,j}-\lambda_2\pa_\nu v_{2,j}\right)\right|_{\pa D_0},
  \enn
  and also $b_1$, $b_2$ are positive constants satisfying the conditions in Lemma \ref{lem3.1}. We see that $\wid h_{1,j}=0$ and $\displaystyle\wid h_{2,j}=(\gamma_2-\gamma_1)v_{2,j}$ on $\Gamma$. By Theorem \ref{thm2.8} we obtain that $\wid g_{1,j}$ and $g_{2,j}$ are uniformly bounded in $L^2(D_0)$. Further, it can be derived from Theorem \ref{thm2.8} that 
  \ben
    \|\wid h_{1,j}\|_{H^{1/2}(\pa D_0)}\;&&= \|v_{1,j}-v_{2,j}\|_{H^{1/2}(\pa D_0\setminus\Gamma)}\\
    &&\leq \|v_{1,j}\|_{H^{1/2}(\pa D_0\setminus\Gamma)}+\|v_{2,j}\|_{H^{1/2}(\pa D_0\setminus\Gamma)}\\
    &&\leq C\left(\|v_{1,j}\|_{H^1(D_0\setminus\ov B)}+\|v_{2,j}\|_{H^1(D_0\setminus\ov B)}\right)\\
    &&\leq C
  \enn
  and
  \ben
    \|\wid h_{2,j}\|_{H^{-1/2}(\pa D_0)}\;&&\leq \|\wid h_{2,j}\|_{H^{-1/2}(\pa D_0\setminus\Gamma)}+\left\|(\gamma_2-\gamma_1)v_{2,j}\right\|_{H^{-1/2}(\Gamma)}\\
    &&\leq C\left(\|v_{1,j}\|_{H^1(D_0\setminus\ov B)}+\|v_{2,j}\|_{H^1(D_0\setminus\ov B)}+\|v_{2,j}\|_{L^p(\Gamma)}\right)\\
    &&\leq C.
  \enn
  Hence, by Lemma \ref{lem3.1} we obtain $\|v_{2,j}\|_{H^1(D_0)}\leq C$ uniformly for $j\in\N$, which contradicts with
  \be\label{4.4}
    \left\|v_{2,j}-\frac{2}{\lambda_2(x_0)+1}\Phi(x,x_j)\right\|_{H^{1}(D\setminus\ov D_b)}\leq C,
  \en
  Thus we get $\lambda_1=\lambda_2$ on $\pa D$.

  Now if  $\left(\pa_\nu\lambda_1-\pa_\nu\lambda_2\right)+2(\gamma_2-\gamma_1)\neq0$ on $\pa D$, we can assume that  $\left(\pa_\nu\lambda_1-\pa_\nu\lambda_2\right)+2(\gamma_2-\gamma_1)>0$ on $\Gamma$. (Note that if $\gamma$ has a imaginary part, we just consider the real or imaginary part of $\left(\pa_\nu\lambda_1-\pa_\nu\lambda_2\right)+2(\gamma_2-\gamma_1)$.) Then we see that $(w_{1,j},w_{2,j})=(\sqrt{\lambda_1}v_{1,j},\sqrt{\lambda_2}v_{2,j})$ satisfies 
  \be\label{3.4a}
  \left\{
  \begin{array}{ll}
    \Delta w_{1,j}-a_1w_{1,j}=0~~~&{\rm in}~D_0,\\
    \Delta w_{2,j}-a_2w_{2,j}=0~~~&{\rm in}~D_0,\\
    w_{1,j}-w_{2,j}=0~~~&{\rm on}~\Gamma,\\
    \displaystyle\pa_\nu w_{1,j}-\pa_\nu w_{2,j}=\left(\pa_\nu\sqrt\lambda_1-\pa_\nu\sqrt\lambda_2\right)v_{2,j}+\frac{\gamma_2-\gamma_1}{\sqrt\lambda_1}v_{2,j}~~~&{\rm on}~\Gamma.
  \end{array}
  \right.
  \en
  Integration by parts yields that 
  \ben
    &&\int_{\pa D_0}\pa_\nu w_{1,j}\ov w_{2,j}ds-\int_{D_0}\na w_{1,j}\cdot\na\ov w_{2,j}dx-\int_{D_0}a_1w_{1,j}\ov w_{2,j}dx=0, \\
    &&\int_{\pa D_0}\pa_\nu w_{2,j}\ov w_{1,j}ds-\int_{D_0}\na w_{2,j}\cdot\na\ov w_{1,j}dx-\int_{D_0}a_2w_{2,j}\ov w_{1,j}dx=0,
  \enn
  which leads to 
  \ben
    \int_{\pa D_0}\left(\pa_\nu w_{1,j}\ov w_{2,j}-\pa_\nu\ov w_{2,j}w_{1,j}\right)ds+\int_{D_0}(\ov a_2-a_1)w_{1,j}\ov w_{2,j}dx=0.
  \enn
  Utilizing the boundary conditions in (\ref{3.4a}), we obtain 
  \be\label{3.5a}\nonumber
    \frac{1}{2}\int_{\Gamma}\left(\pa_\nu\lambda_1-\pa_\nu\lambda_2+2(\gamma_2-\gamma_1)\right)|v_{2,j}|^2ds=&&\;\int_{\Gamma}\left(\pa_\nu\ov w_{2,j}w_{2,j}-\pa_\nu w_{2,j}\ov w_{2,j}\right)ds\\ \nonumber
    &&\;+\int_{\pa D_0\setminus\Gamma}\left(\pa_\nu\ov w_{2,j}w_{1,j}-\pa_\nu w_{1,j}\ov w_{2,j}\right)ds\\
    &&\;+\int_{D_0}(a_1-\ov a_2)w_{1,j}\ov w_{2,j}dx.
  \en
  Further, it can be deduced that 
  \be\label{3.6a}\nonumber
    \int_{\Gamma}\left(\pa_\nu\ov w_{2,j}w_{2,j}-\pa_\nu w_{2,j}\ov w_{2,j}\right)ds=&&\;\left(\int_{\pa D_0}-\int_{\pa D_0\setminus\Gamma}\right)\left(\pa_\nu\ov w_{2,j}w_{2,j}-\pa_\nu w_{2,j}\ov w_{2,j}\right)ds\\ \nonumber
    =&&\;-\int_{\pa D_0\setminus\Gamma}\left(\pa_\nu\ov w_{2,j}w_{2,j}-\pa_\nu w_{2,j}\ov w_{2,j}\right)ds\\
    &&\;+\int_{D_0}(\ov a_2-a_2)|w_{2,j}|^2dx,
  \en
  which indicates that the right hand side of (\ref{3.5a}) is bounded for all $j\in\N$ by Theorem \ref{thm2.8}. Thus we get $\|v_{2,j}\|_{L^2(\Gamma)}\leq C$ uniformly for $j\in\N$ from the assumption, which is a contradiction since $\|\Phi_0(x,x_j)\|_{L^2(\Gamma)}\rightarrow\infty$ as $j\rightarrow\infty$. Therefore,  $\left(\pa_\nu\lambda_1-\pa_\nu\lambda_2\right)+2(\gamma_2-\gamma_1)=0$ on $\pa D$. We leave our proof for $a_1=a_2$ on $\pa D$ to the last. 
  
  (ii)Suppose $D_1\neq D_2$. Since $\gamma_1\in C(\ov D_1)$, we can find a small $C^{2,\alpha_0}$ smooth domain $D_0\subset (D_1\setminus(\overline{D_2\cup D_b^1}))\cap O_1(\pa D_1)\cap O_2(\pa D_1)$ such that $\Gamma:=\pa D_0\cap\pa D_1\neq\emptyset$ is an open subset of $\pa D_1$ and one of the following is satisfied \\
  (a)$\gamma_1=0$ on $\Gamma$, \\
  (b)$\gamma_1>\delta_0$ or $\gamma_1<-\delta_0$ on $\Gamma$ for some positive constant $\delta_0$. \\
  Then there exists a point $x_0\in \Gamma$ and a small ball $B$ centered at $x_0$ such that $D_1\cap B\subset D_0$. Define
  \ben
    x_j:=x_0+\frac{\delta}{j}n(x_0),~~~j=1,2,\cdots
  \enn
  with $\delta>0$ small enough such that $x_j\in B$ for all $j\in\N$. Still denote by $G$ the unbounded connected part of $\R^3\setminus\overline{(D_1\cup D_2)}$. Moreover, by Remark \ref{remark3.4} we can let $D_0$ be sufficiently small such that $(k,n_1,1,\eta,D_0)$ satisfies one of the conditions 1-6 in Theorem \ref{thm3.3}, where $\eta=0$ in case (a) and $\eta=\gamma_1$ in case (b).

  Now consider the incident wave $\wid u^i_j=\na_x\Phi(x,x_j)\cdot\nu(x_0)$. Let  $(\wid u^s_{m}(x;x_j),\wid v_{m}(x;x_j))$ be the unique solution of the transmission problem (\ref{1.1}) corresponding to the scatter $(D_m,\lambda_m,\gamma_m,n_m,D_b^m)$ with the incident wave $u^i_j$, $m=1,2$. Similarly as in the proof of \cite[Theorem 3.1]{FJ18}, by Rellich's Lemma we have $\wid u^s_1(x;x_j)=\wid u^s_2(x;x_j)$, $\forall x\in\ov G$ and $j\in\N$.

  Set $\wid u_{m,j}^s=\wid u_m^s(x;x_j)$, $\wid v_{m,j}=\wid v_m(x;x_j)$ and $\wid u_{m,j}=\wid u^i_j+\wid u_{m,j}^s$ with $m=1,2$. We see that $(\wid v_{1,j},\wid u_{2,j})$ solves the following interior transmission problem
  \be\label{4.7}
  \left\{
  \begin{array}{ll}
    \Delta\wid v_{1,j}+k^2n_1\wid v_{1,j}=0~~~&{\rm in}~D_0,\\
    \Delta\wid u_{2,j}+k^2\wid u_{2,j}=0~~~&{\rm in}~D_0,\\
    \wid v_{1,j}-\wid u_{2,j}=f_1^j~~~&{\rm on}~\pa D_0,\\
    \pa_\nu\wid v_{1,j}+\eta\wid u_{2,j}-\pa_\nu\wid u_{2,j}=f_2^j~~~&{\rm on}~\pa D_0
  \end{array}
  \right.
  \en
  with the boundary data
  \ben
    f_1^j=(\wid v_{1,j}-\wid u_{2,j})|_{\pa D_0},~~~f_2^j=\left.\left(\displaystyle\pa_\nu\wid v_{1,j}+\eta\wid u_{2,j}-\pa_\nu\wid u_{2,j}\right)\right|_{\pa D_0}.
  \enn
  Clearly, we have $f_1^j=f_2^j=0$ on $\Gamma$. By the trace theorem, there exists a function $\chi_2\in H^2(D_0)$ such that $\chi_2=0$ and $\pa\chi_2/\pa\nu=\eta$ on $\pa D_0$. Define $u_0^j=(1-\chi_1)(\wid v_{1,j}-\wid u_{2,j}+\chi_2\wid u_{2,j})$ with
  $\chi_1\in C_0^\infty(\R^3)$ and 
  \begin{align*}
    \chi_1(x)=\left\{
    \begin{array}{ll}
      0, & {\rm in}~\R^3\setminus\ov B, \\
      1, & {\rm in}~B_0,
    \end{array}
    \right.
  \end{align*}
  where $B_0$ is also a small ball centered at $x_0$ such that $\ov B_0\subset B$. Then we see that $u_0^j\in H^1_\Delta(D_0)$ with $u_0^j=f_1^j$ and $\pa u_0^j/\pa\nu=f_2^j$ on $\pa D_0$. Thus by Theorem \ref{thm3.3}, it is obtained that 
  \ben
    \|\wid v_{1,j}\|_{L^2(D_0)}+\|\wid u_{2,j}\|_{L^2(D_0)}\leq C\|u_0^j\|_{H^1_\Delta(D_0)}.
  \enn

  We claim that $\|u_0^j\|_{H^1_\Delta(D_0)}\leq C$ uniformly for $j\in\N$. To this end, first we have $\|\wid u_{2,j}^s\|_{H^1_\Delta(D_0)}\leq C$ because of the well-posedness of problem (\ref{1.1}) and the positive distance between $x_j$ and $D_2$. Then by Theorem \ref{thm2.7} we derive the following
  \ben
    \|u_0^j\|_{H^1(D_0)}\;&&=\|u_0^j\|_{H^1(D_0\setminus\ov B_0)}\\
    &&\leq C\|\wid v_{1,j}-\wid u_{2,j}+\chi_2\wid u_{2,j}\|_{H^1(D_0\setminus\ov B_0)}\\
    &&\leq C(\|\wid v_{1,j}\|_{H^1(D_0\setminus\ov B_1)}+\|\wid u_{2,j}^s\|_{H^1(D_0)}+\|\wid u^i_j\|_{H^1(D_0\setminus\ov B_0)})\\
    &&\leq C,
  \enn
  since $\chi_2\in H^2(D_0)$. Next we show $\Delta u_0^j$ are uniformly bounded in $L^2(D_0)$. It can be verified that
  \ben
    \Delta u_0^j=&&\;\Delta(1-\chi_1)(\wid v_{1,j}-\wid u_{2,j}+\chi_2\wid u_{2,j})+2\nabla(1-\chi_1)\cdot\nabla(\wid v_{1,j}-\wid u_{2,j}+\chi_2\wid u_{2,j})\\
    &&\;+(1-\chi_1)\Delta(\wid v_{1,j}-\wid u_{2,j}+\chi_2\wid u_{2,j}). 
  \enn
  Thus we only need to prove that the last term in above equation is uniformly bounded in $L^2(D_0)$. Since $\wid v_{1,j}$ and $\wid u_{2,j}$ are solutions to problem (\ref{4.7}), we have
  \ben
    \Delta(\wid v_{1,j}-\wid u_{2,j}+\chi_2\wid u_{2,j})\;&&=-k^2n_1\wid v_{1,j}-k^2\wid u_{2,j}+\Delta\chi_2\wid u_{2,j}+2\nabla\chi_2\cdot\nabla \wid u_{2,j}-k^2\chi_2\wid u_{2,j},
  \enn
  which clearly implies $\|(1-\chi_1)\Delta(\wid v_{1,j}-\wid u_{2,j}+\chi_2\wid u_{2,j})\|_{L^2(D_0)}\leq C$ uniformly for $j\in\N$.

  Therefore, we obtain that $\|\wid u_{2,j}\|_{L^2(D_0)}\leq C$, which indicates that
  \ben
    \|\wid u^i_j\|_{L^2(D_0)}\leq \|\wid u_{2,j}\|_{L^2(D_0)}+\|\wid u_{2,j}^s\|_{L^2(D_0)}\leq C.
  \enn
  But by direct calculation we see that
  \ben
    \|\wid u^i_j\|_{L^2(D_0)}^2=\int_{D_0}|\nabla_x\Phi(\cdot,x_j)\cdot\nu(x_0)|^2dx\geq\frac{C}{j^2}\int_{D_0}\frac{1}{|x-x_j|^6}dx=O(j),
  \enn
  which is a contradiction. Hence, we get $D_1=D_2=D$.

  For the proof of the uniqueness of $\gamma$ on $\pa D$,  note that Theorem \ref{thm2.8} is still valid for the case $\lambda=1$ near the boundary $\pa D$, we then follow closely the proof of $\displaystyle\left(\pa\lambda_1/\pa\nu-\pa\lambda_2/\pa\nu\right)+2(\gamma_2-\gamma_1)=0$ on $\pa D$ in (i). The detailed proof is thus omitted here.

  Now assume that $n_1-n_2\neq0$ on $\pa D$.  We let $n_1-n_2>0$ in $\ov D_0$.  It follows that $(\wid v_{1,j},\wid v_{2,j})$ satisfies
  \be\label{4.8}
  \left\{
  \begin{array}{ll}
  	\Delta\wid v_{1,j}+k^2n_1\wid v_{1,j}=0~~~&{\rm in}~D_0,\\
  	\Delta\wid v_{2,j}+k^2n_2\wid v_{2,j}=0~~~&{\rm in}~D_0,\\
  	\wid v_{1,j}-\wid v_{2,j}=0~~~&{\rm on}~\Gamma,\\
  	\pa_\nu\wid v_{1,j}-\pa_\nu\wid v_{2,j}=0~~~&{\rm on}~\Gamma.
  \end{array}
  \right.
  \en
  By Remark \ref{remark3.4}, we can further let $D_0$ be small enough such that $(k,n_1,n_2,0,D_0)$ satisfies the condition 6 in Theorem \ref{thm3.3}. Then estimating analogously as the first part of the proof in (ii), we can also derive a contradiction, which implies that $n_1=n_2$ on $\pa D$.

  Furthermore, from (\ref{4.8}) we notice that 
   \ben
  \left\{
  \begin{array}{ll}
  	\Delta (\wid v_{1,j}-\wid v_{2,j})+k^2n_1(\wid v_{1,j}-\wid v_{2,j})=k^2(n_2-n_1)\wid v_{2,j}~~~&{\rm in}~D_0,\\ 
  	\wid v_{1,j}-\wid v_{2,j}=0~~~&{\rm on}~\Gamma.
  \end{array}
  \right.
  \enn
  By Theorem \ref{thm2.7}, since $n_1=n_2$ on $\Gamma$, we see
  \ben
    \|(n_2-n_1)\wid v_{2,j}\|_{L^2(D_0)}&&\leq C\left(\|\wid v_{2,j}-\wid u^i_j+(\gamma(x_0)/2)\Phi_0(x,x_j)\|_{L^2(D_0)}\right.\\
    &&\quad\quad~+\left.\|(n_2-n_1)(\wid u^i_j-(\gamma(x_0)/2)\Phi_0(x,x_j))\|_{L^2(D_0)}\right)\\
    &&\leq C.
  \enn
  It is known that $\wid v_{1,j},\wid v_{2,j}\in H^2(D_0)$ for every $j\in\N$ since $x_j\in\R^3\setminus\ov D$. Thus, we deduce from \cite[Theorem 9.13]{DN98} that for any domain $\Om'\subset\subset D_0\cup\Gamma$,
  \be\label{4.9}
    \|\wid v_{1,j}-\wid v_{2,j}\|_{H^2(\Om')}\leq C(\|\wid v_{1,j}-\wid v_{2,j}\|_{L^2(D_0)}+\|(n_2-n_1)\wid v_{2,j}\|_{L^2(D_0)})\leq C.
   \en
  We choose $\Om'$ such that $D\cap B\subset\Om'$ and $\Gamma':=\pa\Om'\cap\pa D_0$ is an open subset of $\Gamma$. For simplicity, set $\pa_i:=\pa/\pa x_i$ with $i=1,2,3$. Then for $i=1,2,3$ and $j\in\N$, it follows that 
  \be\label{4.10}
  \left\{
  \begin{array}{ll}
  	\Delta\pa_i\wid v_{1,j}+k^2\pa_in_1\wid v_{1,j}+k^2n_1\pa_i\wid v_{1,j}=0~~~&{\rm in}~\Om',\\ 
  	\Delta\pa_i\wid v_{2,j}+k^2\pa_in_2\wid v_{2,j}+k^2n_2\pa_i\wid v_{2,j}=0~~~&{\rm in}~\Om',\\ 
  	\pa_i\wid v_{1,j}-\pa_i\wid v_{2,j}=0~~~&{\rm on}~\Gamma'.
  \end{array}
  \right.
  \en

  Now we integrate by parts over $\Om'$ and yield that
  \ben
    &&\int_{\pa\Om'}\left(\pa_\nu(\pa_i\wid v_{1,j})\wid v_{2,j}-\pa_i\wid v_{1,j}\pa_\nu\wid v_{2,j}\right)ds=\int_{\Om'}k^2\left(n_2\wid v_{2,j}\pa_i\wid v_{1,j}-\pa_i(n_1\wid v_{1,j})\wid v_{2,j}\right)dx,\\
    &&\int_{\pa\Om'}\left(\pa_\nu(\pa_i\wid v_{2,j})\wid v_{1,j}-\pa_i\wid v_{2,j}\pa_\nu\wid  v_{1,j}\right)ds=\int_{\Om'}k^2\left(n_1\wid v_{1,j}\pa_i\wid v_{2,j}-\pa_i(n_2\wid v_{2,j})\wid v_{1,j}\right)dx,
  \enn
  which implies
  \be\label{4.11}\nonumber
    &&\;\int_{\pa\Om'}\left(\pa_{\nu}(\pa_i\wid v_{1,j})\wid v_{2,j}-\pa_\nu(\pa_i\wid v_{2,j})\wid v_{1,j}\right)ds~\\ \nonumber
    =&&\;\int_{\Om'}k^2\pa_i\left((n_2-n_1)\wid v_{1,j}\wid v_{2,j}\right)dx+\int_{\pa\Om'\setminus\Gamma'}\left(\pa_i\wid v_{1,j}\pa_\nu\wid  v_{2,j}-\pa_i\wid v_{2,j}\pa_\nu\wid v_{1,j}\right)ds~\\
    =&&\;\int_{\pa\Om'\setminus\Gamma'}k^2(n_2-n_1)\wid v_{1,j}\wid v_{2,j}\nu_ids+\int_{\pa\Om'\setminus\Gamma'}\left(\pa_i\wid v_{1,j}\pa_\nu\wid v_{2,j}-\pa_i\wid v_{2,j}\pa_\nu\wid v_{1,j}\right)ds~
  \en
  since $n_1=n_2$ on $\Gamma'$. On the other hand, we have that
  \be\label{4.12}\nonumber
    &&\;\int_{\Gamma'}\left(\pa_{\nu}(\pa_i\wid v_{1,j})-\pa_\nu(\pa_i\wid v_{2,j})\right)\wid v_{1,j}ds\\ \nonumber
    =&&\;\int_{\pa\Om'}\left(\pa_{\nu}(\pa_i\wid v_{1,j})-\pa_\nu(\pa_i\wid v_{2,j})\right)\wid v_{1,j}ds-\int_{\pa\Om'\setminus\Gamma'}\left(\pa_{\nu}(\pa_i\wid v_{1,j})-\pa_\nu(\pa_i\wid v_{2,j})\right)\wid v_{1,j}ds\\ \nonumber
    =&&\;\int_{\pa\Om'\setminus\Gamma'}(\pa_i\wid v_{1,j}-\pa_i\wid v_{2,j})\pa_\nu\wid  v_{1,j}-\int_{\pa\Om'\setminus\Gamma'}\left(\pa_{\nu}(\pa_i\wid v_{1,j})-\pa_\nu(\pa_i\wid v_{2,j})\right)\wid v_{1,j}ds\\
    &&\;+\int_{\Om'}k^2\pa_i[(n_2-n_1)\wid v_{2,j}]\wid v_{1,j}dx+\int_{\Om'}k^2\pa_in_1(\wid v_{2,j}-\wid v_{1,j})\wid v_{1,j}dx.
  \en
  Again integrating by parts, we obtain that
  \be\label{4.13}\nonumber
    &&\;\int_{\Om'}k^2\pa_i[(n_2-n_1)\wid v_{2,j}]\wid v_{1,j}dx\\ \nonumber
    =&&\;\int_{\pa\Om'\setminus\Gamma'}k^2(n_2-n_1)\wid v_{1,j}\wid v_{2,j}\nu_ids+\int_{\Om'}k^2(n_1-n_2)\wid v_{2,j}\pa_i\wid v_{1,j}dx\\ \nonumber
    =&&\;\int_{\pa\Om'\setminus\Gamma'}k^2(n_2-n_1)\wid v_{1,j}\wid v_{2,j}\nu_ids+\int_{\Om'}k^2(n_1-n_2)(\wid v_{2,j}-\wid v_{1,j})\pa_i\wid v_{1,j}dx\\ \nonumber
    &&\;+\frac{1}{2}\int_{\Om'}k^2(n_1-n_2)\pa_i(\wid v_{1,j}^2)dx\\ \nonumber
    =&&\;\int_{\pa\Om'\setminus\Gamma'}k^2(n_2-n_1)\wid v_{1,j}\wid v_{2,j}\nu_ids+\int_{\Om'}k^2(n_1-n_2)(\wid v_{2,j}-\wid v_{1,j})\pa_i\wid v_{1,j}dx\\
    &&\;+\frac{1}{2}\int_{\pa\Om'\setminus\Gamma'}k^2(n_1-n_2)\wid v_{1,j}^2\nu_ids+\frac{1}{2}\int_{\Om'}k^2(\pa_in_2-\pa_in_1)\wid v_{1,j}^2dx.
  \en
  Combinbing (\ref{4.11})-(\ref{4.13}), we finally derive the following equation
  \be\label{4.14}\nonumber
    &&\;\int_{\Om'}k^2(\pa_in_2-\pa_in_1)\wid v_{1,j}^2dx\\ \nonumber
    =&&\;\int_{\pa\Om'\setminus\Gamma'}k^2(n_2-n_1)\wid v_{1,j}^2\nu_ids+2\int_{\Om'}k^2\pa_in_1(\wid v_{1,j}-\wid v_{2,j})\wid v_{1,j}dx\\ \nonumber
    &&\;+2\int_{\Om'}k^2(n_2-n_1)(\wid v_{2,j}-\wid v_{1,j})\pa_i\wid v_{1,j}dx+2\int_{\pa\Om'\setminus\Gamma'}\pa_\nu(\pa_i\wid v_{1,j})(\wid v_{1,j}-\wid v_{2,j})ds\\
    &&\;+2\int_{\pa\Om'\setminus\Gamma'}\pa_i\wid v_{1,j}(\pa_\nu\wid v_{2,j}-\pa_\nu\wid v_{1,j})ds 
  \en
  for $i=1,2,3$. By \cite[Lemma 1.1]{AG02}, Theorem \ref{thm2.7} and (\ref{4.9}), it is deduced that
  \ben
    \left|\int_{\pa\Om'\setminus\Gamma'}\pa_\nu(\pa_i\wid v_{1,j})(\wid v_{1,j}-\wid v_{2,j})ds\right|\;&&\leq \|\pa_\nu(\pa_i\wid v_{1,j})\|_{H^{-3/2}(\pa\Om'\setminus\Gamma')}\|\wid v_{1,j}-\wid v_{2,j}\|_{H^2(\Om')}\\
    &&\leq C\|\pa_i\wid v_{1,j}\|_{H_\Delta(\Om'\setminus\ov B)}\|\wid v_{1,j}-\wid v_{2,j}\|_{H^2(\Om')}\\ 
    &&\leq C\|\wid v_{1,j}\|_{H^1(\Om'\setminus\ov B)}\|\wid v_{1,j}-\wid v_{2,j}\|_{H^2(\Om')}\leq C
  \enn
  and
  \ben
    \left|\int_{\pa\Om'\setminus\Gamma'}\pa_i\wid v_{1,j}(\pa_\nu\wid v_{2,j}-\pa_\nu\wid v_{1,j})\right|\;&&\leq \|\pa_i\wid v_{1,j}\|_{H^{-1/2}(\pa\Om'\setminus\Gamma')}\|\wid v_{1,j}-\wid v_{2,j}\|_{H^2(\Om')}\\
    &&\leq C\|\pa_i\wid v_{1,j}\|_{H_\Delta(\Om'\setminus\ov B)}\|\wid v_{1,j}-\wid v_{2,j}\|_{H^2(\Om')}\leq C.
  \enn
  We can further get that the remained terms in the right hand side of (\ref{4.14}) are uniformly bounded for $j\in\N$ again from Theorem \ref{thm2.7} and (\ref{4.9}), which implies
  \ben
    \left|\int_{\Om'}k^2(\pa_in_2-\pa_in_1)v_{1,j}^2dx\right|\leq C
  \enn
  and thus
  \be\label{4.15}
    \left|\int_{\Om'}(\pa_in_2-\pa_in_1)|\na_x\Phi_0(x,x_j)\cdot\nu(x_0)|^2dx\right|\leq C.
  \en
  But the singularity of $\Phi_0$ yields that
  \ben
    &&\int_{\Om'}|\na_x\Phi_0(x,x_j)\cdot\nu(x_0)|^2dx\geq Cj, \\
    &&\int_{\Om'}|x-x_0||\na_x\Phi_0(x,x_j)\cdot\nu(x_0)|^2dx\geq C\ln j
  \enn
  as $j\rightarrow\infty$ with the constant $C>0$ independent of $j\in\N$. Therefore, the regularity of $n$ (i.e., $n_m\in C^2(\ov D\setminus D_b^m)$, $m=1,2$) toghther with (\ref{4.15}) forces $D^\alpha n_1=D^\alpha n_2$ on $\pa D$ for $|\alpha|=1,2$. 
  
  (iii)We continue the proof of part (i) for $a_1=a_2$ on $\pa D$. Since $\left(\pa_\nu\lambda_1-\pa_\nu\lambda_2\right)+2(\gamma_2-\gamma_1)=0$ on $\pa D$, we see that (\ref{3.4a}) becomes
  \be\label{4.16}
  \left\{
  \begin{array}{ll}
  	\Delta w_{1,j}-a_1w_{1,j}=0~~~&{\rm in}~D_0,\\
  	\Delta w_{2,j}-a_2w_{2,j}=0~~~&{\rm in}~D_0,\\
  	w_{1,j}-w_{2,j}=0~~~&{\rm on}~\Gamma,\\
  	\pa_\nu w_{1,j}-\pa_\nu w_{2,j}=0~~~&{\rm on}~\Gamma.
  \end{array}
  \right.
  \en
  Similar to (\ref{4.9}), by Theorem \ref{thm2.8} and $\lambda_1=\lambda_2$ on $\pa D$ we can also get $\|w_{1,j}-w_{2,j}\|_{H^2(\Om')}\leq C$. Following the similar procedure from (\ref{4.11}) to (\ref{4.14}), we have
  \be\label{4.17}\nonumber
    &&\;\int_{\pa\Om'}(a_1-a_2)w_{1,j}^2\nu_ids\\ \nonumber
    =&&\;\int_{\Om'}(\pa_ia_1-\pa_ia_2)w_{1,j}^2dx+2\int_{\Om'}\pa_ia_1(w_{1,j}-w_{2,j})w_{1,j}dx\\ \nonumber
    &&\;+2\int_{\Om'}(a_2-a_1)(w_{2,j}-w_{1,j})\pa_iw_{1,j}dx+2\int_{\pa\Om'\setminus\Gamma'}\pa_\nu(\pa_iw_{1,j})(w_{1,j}-w_{2,j})ds \\
    &&\;+2\int_{\pa\Om'\setminus\Gamma'}\pa_iw_{1,j}(\pa_\nu w_{2,j}-\pa_\nu w_{1,j})ds
  \en
  with $i=1,2,3$. Again we can derive that
  \ben
    \left|\int_{\pa\Om'}(a_1-a_2)w_{1,j}^2\nu_ids\right|\leq C,
  \enn
  which indicates by Theorem \ref{thm2.8} that
  \ben
    \left|\int_{\Gamma'}(a_1-a_2)\lambda_1|\Phi_0(x,x_j)|^2\nu_ids\right|\leq C.
  \enn
  Now suppose $(a_1-a_2)(x_0)\neq 0$ for $x_0\in\pa D$. Then we can assume that $a_1-a_2>0$ in $\ov \Om'$ (or the real or imaginary part of $a_1-a_2$).  We know that there exists $i\in\{1,2,3\}$ such that $\nu_i(x_0)\neq 0$. Hence, we can further assume that $\nu_i\geq\varepsilon>0$ on $\Gamma'$, which implies $\|\Phi_0(x,x_j)\|_{L^2(\Gamma')}\leq C$. This is however a contradiction. Therefore, $a_1=a_2$ on $\pa D$, and the proof is finally finished.
\end{proof}
\begin{remark}\label{remark3.8}
  Note that the regularity requirements on the coefficients in this theorem can be relaxed to only in the neighborhood of the boundary $\pa D$. Thus if $\lambda$ is a constant near the boundary $\pa D$, we immediately get the uniqueness of $\gamma$ and $n$ at the boundary. Furthermore, if $\lambda\equiv1$ and $n$ is a polynomial of degree {\rm 2}, then by the standard arguments we can recover $D_b$ and $n$ simultaneously. We also note that there is a possibility to extend our methods to recover $D_b$ and analytical $n$. 
\end{remark}

\section*{Acknowledgements} This work is partly supported by NSF of China under No. 12122114.


\begin{thebibliography}{99}
\bibitem{AJ03} A. Adams and J.F. Fournier, {\em Sobolev Spaces} (2nd Ed.), Elsevier, Singapore,2003.

\bibitem{AG02} A. L. Bukhgeim and G. Uhlmann, Recovering a potential from partial Cauchy data, {\em Comm. Partial Differential Equations \bf27}(2002), 653-668.

\bibitem{FD03} F. Cakoni and D. Colton, A uniqueness theorem for an inverse electromagnetic scattering problem in inhomogeneous anisotropic meida, {\em Proc. Edinb. Math. Soc. \bf46}(2003), 293-314. 

\bibitem{FD06} F. Cakoni and D. Colton, {\em Qualitative Methods in Inverse Scattering Theory}, Springer, Berlin, 2006. 

\bibitem{DR13} D. Colton and R. Kress, {\em Inverse Acoustic and Electromagnetic Scattering Theory} (3rd Ed.),
Springer, New York, 2013.

\bibitem{DRP97} D. Colton, R. Kress and P. Monk, Inverse scattering from an orthotropic medium, {\em J. Comput. Appl. Math. \bf81}(1997), 269-298.

\bibitem{TR96} T. Gerlach and R. Kress, Uniqueness in inverse obstacle scattering with conductive boundary condition, {\em Inverse Probl. \bf12}(1996), 619-625.

\bibitem{DN98} D. Gilbarg and N.S. Trudinger, {\em Elliptic Partial Differential Equations of Second Order} (2nd Ed.), Springer, New York, 1998.

\bibitem{PH93} P. H\"{a}hner, A uniqueness theorem for a transmission problem in inverse electromagnetic scattering, {\em Inverse Probl. \bf9}(1993), 667-678. 

\bibitem{PH98} P. H\"{a}hner, A uniqueness theorem for an inverse scattering problem in an exterior domain, {\em SIAM J. Math. Anal. \bf29}(1998), 1118-1128. 

\bibitem{FH94} F. Hettlich, On the uniqueness of the inverse conductive scattering problem for the Helmholtz equation, {\em Inverse Probl. \bf10}(1994), 129-144.
    
\bibitem{FH96} F. Hettlich, Uniqueness of the inverse conductive scattering problem for time-harmonic electromagnetic waves, {\em SIAM J. Appl. Math. \bf56}(1996), 588-601.

\bibitem{VI90} V. Isakov, On uniqueness in the inverse transmission scattering problem, {\em Comm. Partial Differential Equations \bf15}(1990), 1565-1587.

\bibitem{VI08} V. Isakov, On uniqueness in the general inverse transmission scattering problem, {\em Comm. Math. Phys. \bf280}(2008), 843-858.

\bibitem{AR93} A. Kirsch and R. Kress, Uniqueness in inverse obstacle scattering, {\em Inverse Probl. \bf9}(1993), 285-299.

\bibitem{AL98} A. Kirsch and L. P\"{a}iv\"{a}rinta, On recovering obstacles inside inhomogeneities, {\em Math. Methods Appl. Sci. \bf21}(1998), 619-651. 

\bibitem{XB10} X. Liu and B. Zhang, Direct and inverse obstacle scattering problems in a piecewize homogeneous medium, {\em SIAM J. Appl. Math. \bf70}(2010), 3105-3120.

\bibitem{XBJ10} X. Liu, B. Zhang and J. Yang, The inverse electromagnetic scattering problem in a piecewize homogeneous medium, {\em Inverse Probl. \bf26}(2010)125001, 19pp. 

\bibitem{DM98} D. Mitrea and M. Mitrea, Uniqueness for inverse conductivity and transmission problems in the class of Lipschitz domains, {\em Comm. Partial Differential Equations \bf23}(1998), 1419-1448.

\bibitem{RP98} R. Potthast, A point-source method for inverse acoustic and electromagnetic obstacle scattering problems, {\em IMA J. Appl. Math. \bf61}(1998), 119-140. 

\bibitem{RP01} R. Potthast, On the convergence of a new Newton-type method in inverse scattering, {\em Inverse Probl. \bf17}(2001), 1419-1434. 

\bibitem{FJ18} F. Qu and J. Yang, On recovery of an inhomogeneous cavity in inverse acoustic scattering, {\em Inverse Probl. Imaging \bf12}(2018), 281-291.

\bibitem{FJB17} F. Qu, J. Yang and B. Zhang, Recovering an elastic obstacle containing embedded objects
by the acoustic far-field measurements, {\em Inverse Probl. \bf34}(2017)015002, 8pp.

\bibitem{NV04} N. Valdivia, Uniqueness in inverse obstacle scattering with conductive boundary conditions, {\em Applic. Anal. \bf83}(2004), 825-853.

\bibitem{JY21} J. Yang, Determing conductivity and embedd obstacles from partial boundary measurements, (arXiv:2104.13552v1).

\bibitem{JB11} J. Yang and B. Zhang, An inverse scattering problem for periodic media, {\em Inverse Probl. \bf27}(2011)125010, 22pp. 

\bibitem{JBH18} J. Yang, B. Zhang and H. Zhang, Uniqueness in inverse acoustic and electromagnetic scattering
by penetrable obstacle with embedded objects, {\em J. Differ. Equations \bf265}(2018), 6352-6383.	
\end{thebibliography}
\end{document}